\documentclass{amsart}

\usepackage[utf8]{inputenc}
\usepackage[top=1.25in, bottom=1.25in, left=1.1in, right=1.1in]{geometry, mathtools}
\usepackage[all]{xy}
\usepackage[toc,page]{appendix}
\usepackage{amsthm, amsmath, amssymb, amscd, latexsym, multicol, verbatim, enumerate, graphicx,xy, color}
\usepackage{bm}
\usepackage{adjustbox}
\usepackage{mathrsfs}
\usepackage{tikz,stackrel}
\usepackage{mathdots}
\usepackage[all]{xy}
\usetikzlibrary{calc}
\usetikzlibrary{matrix,arrows,decorations.pathmorphing}
\usetikzlibrary{intersections}
\usetikzlibrary{decorations.markings}
\usetikzlibrary{external}
\usepackage{setspace}
\usepackage{multirow}

\usepackage{tikz-cd}
\usetikzlibrary{decorations.markings}
\makeatletter
\tikzcdset{
open/.code={\tikzcdset{hook, circled};},
closed/.code={\tikzcdset{hook, slashed};},
circled/.code={\tikzcdset{markwith={\draw (0,0) circle (.375ex);}};},
slashed/.code={\tikzcdset{markwith={\draw[-] (-.4ex,-.4ex) -- (.4ex,.4ex);}};},
markwith/.code={
\pgfutil@ifundefined{tikz@library@decorations.markings@loaded}%
{\pgfutil@packageerror{tikz-cd}{You need to say %
\string\usetikzlibrary{decorations.markings} to use arrow with markings}{}}{}%
\pgfkeysalso{/tikz/postaction={/tikz/decorate,
/tikz/decoration={
markings,
mark = at position 0.5 with
{#1}}}}},
}
\makeatother
\makeatletter
\tikzset{
circled/.code={\tikzset{markwith={\draw (0,0) circle (.375ex);}};},
slashed/.code={\tikzset{markwith={\draw[-] (-.4ex,-.4ex) -- (.4ex,.4ex);}};},
markwith/.code={
\pgfutil@ifundefined{tikz@library@decorations.markings@loaded}%
{\pgfutil@packageerror{tikz-cd}{You need to say %
\string\usetikzlibrary{decorations.markings} to use arrow with markings}{}}{}%
\pgfkeysalso{/tikz/postaction={/tikz/decorate,
/tikz/decoration={
markings,
mark = at position 0.5 with
{#1}}}}},
}
\makeatother

\makeatletter
\providecommand{\leftsquigarrow}{%
  \mathrel{\mathpalette\reflect@squig\relax}%
}
\newcommand{\reflect@squig}[2]{%
  \reflectbox{$\m@th#1\rightsquigarrow$}%
}
\makeatother

\usepackage[english]{babel}
\usepackage{stmaryrd}
\usepackage{stackrel}
\usepackage{mathbbol}
\usepackage{float}
\usepackage{comment}
\usepackage[colorlinks=true, pdfstartview=FitV, linkcolor=blue, citecolor=blue, urlcolor=blue, breaklinks=true]{hyperref}
\usepackage{theoremref}
\usepackage{caption}
\usepackage{youngtab}
\usepackage{ragged2e}
\usepackage{relsize}
\usepackage{anyfontsize}
\usepackage[utf8]{inputenc}
\usepackage{bm}
\usepackage{scalerel}
\usepackage{marvosym}
\DeclareFontFamily{U}{mathx}{\hyphenchar\font45}
\DeclareFontShape{U}{mathx}{m}{n}{
      <5> <6> <7> <8> <9> <10>
      <10.95> <12> <14.4> <17.28> <20.74> <24.88>
      mathx10
      }{}
\DeclareSymbolFont{mathx}{U}{mathx}{m}{n}
\DeclareFontSubstitution{U}{mathx}{m}{n}
\DeclareMathAccent{\widecheck}{0}{mathx}{"71}
\DeclareMathAlphabet{\mathbbm}{U}{bbm}{m}{n}
\usepackage{verbatim}
\usepackage{pdflscape}

\usepackage{array} 
\newcolumntype{C}{>{$}c<{$}}

\makeatother
\theoremstyle{remark}
\numberwithin{equation}{section}

\theoremstyle{definition}
\newtheorem*{theorem*}{Theorem}
\newtheorem*{definition*}{Definition}
\newtheorem{theorem}{Theorem}[section]
\newtheorem{definition}[theorem]{Definition}

\newtheorem*{proposition*}{Proposition}
\newtheorem{lemma}[theorem]{Lemma}
\newtheorem{corollary}[theorem]{Corollary}
\newtheorem{remark}[theorem]{Remark}

\newtheorem{example}[theorem]{Example}
\newtheorem*{example*}{Example}

\newcommand{\lrp}[1]{\left(#1\right)}

\newcommand{\Q}{\mathbb{Q} }
\newcommand{\R}{\mathbb{R} }

\newcommand{\Z}{\mathbb{Z} }

\makeatletter
\newcommand{\oversetcustom}[3][0ex]{%
  \mathrel{\mathop{#3}\limits^{
    \vbox to#1{\kern-0.5\ex@
    \hbox{$\scriptstyle#2$}\vss}}}}
\makeatother

\newcommand{\T}{\mathbb{T}}

\newcommand{\gv}{\mathbf{g} }
\newcommand{\cv}{\mathbf{c} }

\newcommand{\cA}{\mathcal{A} }
\newcommand{\cAp}{\mathcal{A}_{\mathrm{prin}}}
\newcommand{\cXp}{\mathcal{X}_{\mathrm{prin}}}

\newcommand{\cX}{\mathcal{X} }
\newcommand{\cV}{\mathcal{V} }

\newcommand{\tf}{\vartheta }

\newcommand{\Iuf}{I_{\text{uf}}}

\newcommand{\udim}{\underline{\dim}}

\newcommand{\seed}{\textbf{s}}

\DeclareMathOperator{\proj}{proj}
\DeclareMathOperator{\inj}{inj}

\newcommand{\KKb}{\mathsf{K}^{\rm b}}

\newsavebox{\mybox}
\newcommand{\back}[1]{%
  \ThisStyle{\ifmmode%
    \savebox{\mybox}{$\SavedStyle#1$}%
    \reflectbox{\usebox{\mybox}}%
  \else%
    \savebox{\mybox}{#1}%
    \reflectbox{\usebox{\mybox}}%
  \fi%
}}

\DeclareMathOperator{\Ext}{Ext}

\DeclareMathOperator{\Hom}{Hom}

\DeclareMathOperator{\Spec}{Spec}

\newcommand{\rank}{\operatorname{rank}}

\newcommand{\ie}{{\em i.e. }}

\DeclareFontFamily{U}{musix}{}
\DeclareFontShape{U}{musix}{m}{n}{<-> s*[1.01] musix11}{}

\newcommand{\rmesh}{\overrightarrow{\text{mesh}}}

\title[The cluster complex for cluster $\cX$-varieties and representations of acyclic quivers]{The cluster complex for cluster Poisson varieties and representations of acyclic quivers}
\author{Carolina Melo and Alfredo N\'ajera Ch\'avez}

\begin{document}

\maketitle

\begin{abstract}
Let $\mathcal{X}$ be a skew-symmetrizable cluster Poisson variety. The cluster complex $\Delta^+(\mathcal{X})$ was introduced in \cite{GHKK}. It codifies the theta functions on $\mathcal{X}$ that restrict to a character of a seed torus. 
Every seed ${ \bf s}$ for $\mathcal{X}$ determines a fan realization $\underline{\Delta^+_{\bf s}(\mathcal{X})}$ of $\Delta^+(\mathcal{X})$.
For every $\seed$ we provide a simple and explicit description of the cones of $\underline{\Delta^+_{\bf s}(\mathcal{X})}$ and their facets using {\bf c}-vectors.
Moreover, we give formulas for the theta functions parametrized by the integral points of $\underline{\Delta^+_{\bf s}(\mathcal{X})}$ in terms of $F$-polynomials.

In case $\mathcal{X}$ is skew-symmetric and the quiver $Q$ associated to ${\bf s}$ is acyclic, we describe the normal vectors of the supporting hyperplanes of the cones of $\underline{\Delta^+_{\bf s}(\mathcal{X})}$ using {\bf g}-vectors of (non-necessarily rigid) objects in $\mathsf{K}^{\rm b}(\text{proj} \; kQ)$.
\end{abstract}

\tableofcontents

\section{Introduction}
\subsection{Overview} Cluster complexes are a class of simplicial complexes that naturally arise in the theory of cluster algebras.
They were introduced by Fomin and Zelevinsky in \cite{FZ_II} as an important tool in the classification of cluster algebras of finite cluster type. 
Since then, many authors have considered these complexes as they turned out to be very relevant in various contexts.
In particular, the work of Fock--Goncharov \cite{FG_cluster_ensembles} and of Gross-Hacking-Keel-Kontsevich \cite{GHKK} show the importance of cluster complexes in the study of cluster varieties and their rings of regular functions.

Throughout this introduction we use the standard terminology employed in the theory of cluster algebras and cluster varieties, the precise definitions will be recalled throughout the text.

Cluster varieties are schemes that can be obtained gluing a (possibly infinite) collection of algebraic tori --the \emph{seed tori}-- using distinguished birational maps called \emph{cluster transformations}.
The gluing is governed by a combinatorial framework that is completely determined by the choice of an \emph{initial seed}.
In \cite{GHKK}, Gross, Hacking, Keel and Kontsevich introduced trendsetting techniques for the study of cluster algebras and cluster varieties. In particular, they defined theta functions on cluster varieties and showed that cluster complexes allow to identify those theta functions that restrict to a character of a particular seed torus. There are two kinds of cluster varieties: cluster $\cA$-varieties, also known as cluster $K_2$ varieties, and cluster $\cX$-varieties, also known as cluster Poisson varieties.
From the perspective of \cite{GHKK}, every cluster variety $\cV$ (either of type $\cA$ or $\cX$) has an associated \emph{(Fock--Goncharov) cluster complex} $\Delta^+(\cV)$.
Every choice of seed $\seed$ gives rise to a fan realization $\underline{\Delta^+_\seed(\cV)}$ of $\Delta^+(\cV)$. 
From this point of view the usual cluster complexes introduced by Fomin and Zelevinsky are associated to the cluster $\cA$-varieties and the corresponding fan realizations are the well-known \emph{{\bf g}-vector fans}.
%The cluster complex associated to a cluster $\cA$-varieties are simplicial complex and the $\gv$-vector fans are simplicial fans.In contrast, the fan $\Delta^+_\seed(\cX)$ in general is not simplicial, hence $\Delta^+(\cX)$ does not gives rise to a simplicial complex, however, it gives rise to a CW complex.
The aim of this paper is to initiate a systematic study of the cluster complexes associated to cluster Poisson varieties \emph{without frozen directions}\footnote{Unlike the $\cA$ case, the cluster complex associated to a cluster Poisson variety can change if we add frozen directions, see Remark~\ref{rem:frozen_directions}.}. 

\subsection{The skew-symmetrizable case} In order to state our description of the cluster complex associated to a cluster Poisson variety we need to recollect some well known concepts in the theory of cluster algebras. So let $\cA$ and $\cX$ be a pair of skew-symmetrizable cluster varieties associated to the same initial seed $\seed$.
%Our description of $\Delta^+_{\seed}(\cX)$ is in terms of the {\bf c}- and {\bf g}-matrices.
Let $B_{\seed}$ be the $n\times n$ skew-symmetrizable integer matrix associated to $\seed$. 
We consider the (rooted) $n$-regular tree $\mathbb T_n$. For simplicity we identify the vertices of $\mathbb T_n$ with the seeds mutation equivalent to $\seed$, let $\seed$ be the root (\ie a distinguished vertex) of $\T_n$ and write $ \seed'\in \T_n$ to indicate that $\seed'$ is a vertex of $\T_n$.
Associated to every $n\times n$ skew-symmetrizable matrix there are two families of $n\times n$ integral matrices called the $\cv$- and $\gv$-matrices, see \cite{NZ}.
The elements of these families are indexed by the vertices of $\mathbb T_n$. 
Given an initial skew-symmetrizable matrix $B$ and $\seed'\in \mathbb T_n$ we let $C^{B}_{\seed'}$ (resp. $G^B_{\seed'}$) denote the $\cv$-matrix (resp. $\gv$-matrix) indexed by the vertex $\seed'$ of $\mathbb T_n$ (taking $B$ as the initial matrix).
In particular, we consider the skew-symmetrizable matrices $B_\seed$ and $-B^T_{\seed}$, where $X^T$ denotes the transpose of a matrix $X$.
The columns of $C^{-B_{\seed}^T}_{\seed'}$ (resp. $G^{B_\seed}_{\seed'}$) are $c_{1;\seed'}, \dots ,c_{n;\seed'}$ (resp. $g_{1;\seed'}, \dots ,g_{n;\seed'}$) and the vectors that arise in this way are the $\cv$-vectors associated to $-B^T_\seed$ (resp. {\bf g}-vectors associated to $B_\seed$).
Let $N$ (resp. $M^\circ$) be the character lattice of the seed tori in the atlas of $\cX $ (resp. $\cA$).
In particular the $\cv$-vectors (resp. $\gv$-vectors) associated to $-B^T_\seed$ (resp. $B_\seed$) belong to $N$ (resp. $M^\circ$).
Let $N_{\R}= N\otimes \R$, $M^\circ_{\R}=M^\circ \otimes \R$ and $\mathcal{G}_{\seed'}$ be the cone in $M^\circ_{\R}$ spanned by the $\gv$-vectors in $G^{B_{\seed}}_{\seed'}$.
Finally, consider the cluster ensemble lattice map $p^*:N \to M^\circ$ associated to $\cA$ and $\cX$. 
The ambient vector space for $\underline{\Delta^+_{\seed}(\cA)}$ (resp. $\underline{\Delta^+_{\seed}(\cX)}$) is $M^\circ_{\R}$ (resp. $N_{\R}$). 
The following summarizes Lemma~\ref{cone}, Corollary~\ref{conedescription} and Lemma~\ref{initiallemma} bellow.

\begin{theorem*}
    $(1)$ For every $\seed'\in \mathbb T_n$ the cone $(p^*)^{-1}(\mathcal{G}_{\seed'})$ is a cone of the fan $\underline{\Delta^+_{\seed}(\cX)}$ and each of its cones arises in this way. Moreover, if for $\beta\in N_{\R}$ we let $\beta_\seed$ be the column vector obtained by writing $\beta$ in the basis of $N_{\R}$ determined by $\seed$ then
    \[
(p^*)^{-1}(\mathcal G_{\seed'})=\{\beta \in N_\R \; | \; (C_{\seed'}^{-B_{\seed}^T})^TB_{\seed} \beta_{\seed} \geq \textbf{0}\}.
\]
$(2)$ If $p^*(\beta) \in \mathcal{G}_{\seed'}$ then the theta function parametrized by $\beta_{\seed}=(\beta_1,\dots, \beta_n)$ is given by
\[
\vartheta_{\beta}^{\mathcal{X}}:=\prod_{i=1}^{n} y_i^{\beta_i}(F_{i;\seed'}(y_1, \ldots, y_n))^{c^T_{i;\seed'}B_{\seed} \beta_{\seed}},
\]
where $y_1,\dots, y_n$ are the coordinates of the seed torus of $\cX$ determined by $\seed$ and $F_{i;\seed'}$ is the $F$-polynomial associated $i$ and $\seed'$ (taking $B_\seed$ as initial exchange matrix).
\end{theorem*}

The $\gv$-vectors in a $\gv $-matrix associated to $B_{\seed}$ span a maximal cone of $\underline{\Delta^+_\seed(\cA)}$. Item $(1)$ of the theorem above tells us that every $\cv$-matrix also determines a (not necessarily maximal) cone of $\underline{\Delta^+_{\seed}(\cX)} $, however, in a different way. 
That is, the $\cv$-vectors of a $\cv$-matrix in general \emph{do not} span a cone of $\underline{\Delta^+_{\seed}(\cX)} $ but such $\cv$-vectors determine the equations of the supporting hyperplanes of a cone of $\underline{\Delta^+_{\seed}(\cX)} $ (the dimension of $(p^*)^{-1}(\mathcal{G}_{\seed'})$ depends on $\seed'$).
Standard results in polyhedral geometry allow us to describe the dimension of $(p^*)^{-1}(\mathcal{G}_{\seed'})$ using the \emph{implicit equalities} of the system $(C_{\seed'}^{-B_{\seed}^T})^TB_{\seed} \beta_{\seed} \geq \textbf{0} $, see equation \eqref{eq:dim}. In particular, one can search for positive linear combinations of the {\bf c}-vectors $c^T_{1;\seed'}, \dots c^T_{n;\seed'}$ that lie in the kernel of $p^*$ in order to determine the implicit equalities of the system. 

\subsection{The skew-symmetric acyclic case} The supporting hyperplane of the cone $(p^*)^{-1}(\mathcal G_{\seed'})$ determined by the $\cv$-vector $c_{j;\seed'}^T$ is defined by the equation $c_{j;\seed'}^TB_\seed \beta_\seed=0$. 
 Writing the elements of $M^\circ$ in the basis determined by $\seed$, the normal vector of such hyperplane is $c_{j;\seed'}^T B_\seed \in M^\circ$. In the skew-symmetric acyclic case we are able to describe this normal vector using $\gv$-vectors (from the perspective of silting theory) and Auslander-Reiten theory. We proceed to explain this.

We assume from now on that the quiver $Q$ associated to $ B_\seed$ is acyclic and consider the path algebra $kQ$. 
Since we are in the skew-symmetric case we have that $M^\circ=M=\Hom (N,\Z)$.
%Let $\{e_1, \dots e_n\} $ be the basis of $N$ determined by $ \seed$ and $\{e_1^*, \dots, e_n^*\}$ the dual basis.
Let $D^b_{kQ}$ be the derived category of the category $\text{mod}\ kQ$ of right $kQ$-modules and $\KKb(\proj kQ)$ be the homotopy category of bounded complexes of finitely generated projective $kQ$-modules.
We consider the Grothendieck groups $K_0 (D^b_{kQ})$ and $K_0(\KKb(\proj kQ))$ and think of the $kQ$-modules (resp. projective $kQ$-modules) as stalk complexes in $D^b_{kQ}$ (resp. in $\proj kQ$) concentrated in degree $0$. 

Recall that every $\cv$-vector is non-zero and is either positive (all its entries are non-negative) or negative (all its entries are non-positive). By \cite{Najera,Nagao}, we have that the positive $\cv$-vectors are in bijection with the dimension vectors of the rigid indecomposable $kQ$-modules. The lattice of dimension vector of objects in $D^b_{kQ}$ is  $K_0(D^b_{kQ})$ (we think of dimension vectors as row vectors).
Hence, every $\cv$-vector can be thought of as the (transpose of a) dimension vector of a stalk complex of $D^b_{kQ}$ (concentrated in degree $0$ if it is positive and in degree $-1$ if it is negative).
Similarly, every object $X$ in $K_0(\KKb(\proj kQ))$ has an associated $\gv$-vector $\gv^X_P$ (with respect to the silting complex $P=kQ$), see \cite{DIJ}.
We can think of $\gv^X_P$ as the class of $X$ in $K_0(\KKb(\proj kQ))$ (the actual integer vector is obtained expressing the class of $X$ in the basis of $K_0(\KKb(\proj kQ))$ given by the indecomposable projective $kQ$-modules).
We have the following result.

\begin{theorem*}[Theorem~\ref{dv-gv}]
    Let $X$ be an indecomposable object of $D^b_{kQ}$, then
    \[B_\seed  (\udim \; X)^T=-\gv^X_P -\gv^{\tau^{-1}X}_P=-\sum_{X_i\in \rmesh_X} \gv^{X_i}_P.\]
    where $\udim \ X$ is the dimension vector of $X$ and the sum runs over the modules $X_i$ in the middle of the mesh (in the Auslander-Reiten quiver of $D^b_{kQ}$) starting in $X$ and ending in $\tau^{-1}(X)$.
\end{theorem*}
In particular, if $   c_{j;\seed'}^T=\udim \ X$ then (using the fact that $B_\seed $ is skew-symmetric) we obtain that 
\[
c_{j;\seed'}^T  B_\seed = \sum_{X_i\in \rmesh_X} {\bf g}_P^{X_i}.
\]

\subsection{Organization}

The structure of the paper is the following. In \S~\ref{sec:cluster_complex} we provide the background on cluster varieties and set up the notation that will be used throughout the text. For simplicity we restrict our attention to the skew-symmetric case as the notation in this case is simpler. 
%material about cluster complexes for cluster Poisson varieties, establishing the groundwork for the subsequent discussion. 
In \S~\ref{sec:Desc_cluster_complex} we give the description of $\Delta^+_{\seed}(\cX) $ in terms of $\cv$-vectors. The description is given in the skew-symmetric case and in \S~\ref{skew-symmetrizable} we indicate how to generalize the results from the skew-symmetric to the skew-symmetrizable case. 
In \S~\ref{sec:theta_functions} we provide the formulas expressing the theta functions on $\mathcal{X}$ parametrized by the integral points of $\Delta^+_{\seed}(\cX)$ using $F$-polynomials. 
In \S~\ref{sec:supporting_hyperplanes} we describe the normal vectors of the supporting hyperplanes of the cones of $\Delta^+_\seed(\cX)$ using $\gv$-vectors provided the quiver $Q$ associated to $\seed$ is acyclic.

\section{Background on cluster varieties}
\label{sec:cluster_complex}

\subsection{Cluster varieties}
\label{sec:cv}
The reader is referred to \cite[\S2]{GHK_birational} and \cite{FG_cluster_ensembles} for a detailed account on cluster varieties, to \cite[\S3]{GHKK}, \cite[\S1.1]{FG_cluster_ensembles} and \cite[\S2.2]{BCMNC} for background on tropicalization of cluster varieties and to \cite[\S3]{GHKK} (resp. \cite[\S7.2]{GHKK}) for the details concerning the Fock--Goncharov cluster complex associated to cluster $\cA$-varieties (resp. cluster Poisson varieties). For simplicity throughout the paper we focus on the skew-symmetric case, however, some of our main results hold in the skew-symmetrizable case as explained in \S~\ref{skew-symmetrizable} and Remark~\ref{rem:tf_sk}.

We begin by setting the notation related to cluster varieties. Throughout this work we fix an algebraically closed field $k$ of characteristic $0$\footnote{Part of the framework in this paper related to representation theory does not require that $k$ is algebraically closed nor of characteristic $0$, however, these hypotheses are needed for the framework related to cluster varieties.}.
 The split torus over $k$ associated to a lattice $L $ is the scheme $T_L:= \Spec(k[L^*])$, where $L^*$ is the dual lattice to $L$. 

Recall (for instances from \cite{GHK_birational}) that in order to construct a cluster variety we need a \emph{fixed data} and a \emph{seed}.
In this section we introduce skew-symmetric cluster varieties without frozen directions. Hence the fixed data $\Gamma$ we consider consists of the following:
\begin{itemize}
    \item the set $I=I_{\rm uf}=\{1,\dots,n\}$;
    \item a lattice $N\cong \Z^n$; 
    \item the dual lattice $M = \Hom (N, \Z)$;
    \item a skew-symmetric bilinear form $\{ \cdot , \cdot \} : N \times N \rightarrow \Z$.
\end{itemize}
A seed is a tuple $\seed := ( e_i )_{i \in I}$ such that $\{ e_i \}_{i\in I}$ is a basis for $N$.
We let $\{ e^*_i \}_{i\in I}$ be the basis of $M$ dual to $\{ e_i \}_{i\in I}$. 
For $i,j\in I$ we let $b_{ij}:=\{e_j,e_i\}$ and consider the $n\times n $ skew-symmetric matrix
\[
B_{\seed}= (b_{ij}).
\]

The quiver associated to $B_\seed$ is denoted by $Q_{\seed}$. Its set of vertices is $I$ and the arrows between $i$ and $j$ are $|b_{ij}|$, where the arrows go from $i $ to $j$ if $b_{ij}$ is positive and from $j $ to $i$ if $b_{ij}$ is negative.

We fix a seed $ \seed$ and call it the initial seed.
Associated to $\Gamma$ and $\seed$ there are four cluster varieties (see \cite[Appendix A]{GHKK}): $\cA$, $\cX$, $\cAp$ and $\cXp$. 
These are schemes over $k$ although $\cAp$ has a very useful structure of a scheme over a Laurent polynomial ring, see \cite[\S3]{BFMNC} for further details. 
%We fix $Q$ once and for all and denote these schemes simply by $ \cA$, $\cX$, $\cAp$ and $\cXp$, respectively.
The dimension of $\cA$ and $\cX$ is $n$ whereas the dimension of $\cAp$ and $\cXp$ is $2n$. Moreover, these schemes are endowed with atlases of tori of the form
\[
\cA= \bigcup_{\seed'\in \T_n}T_{N, \seed'}, \quad \quad \cX= \bigcup_{\seed'\in \T_n}T_{M, \seed'},  \quad \quad \cAp= \bigcup_{\seed'\in \T_n}T_{N\oplus M, \seed'}, \quad \text{and}   \quad \cXp= \bigcup_{\seed'\in \T_n}T_{M\oplus N, \seed'},
\]
where the unions are taken over the vertices of the $n$-regular tree $\T_n$ and for every lattice $L$ the symbol $T_{L,\seed'} $ denotes a copy of $T_{L}$ indexed by $\seed'$. 
Given a cluster variety $\cV$, we denote by $\cV_{\seed'}$ the tori in its atlas parametrized by $\seed'$. 
For instances, $\cA_{\seed'}=T_{N,\seed'}$.
The precise way these tori are glued to each other is described in \cite[\S2]{GHK_birational}.
In fact, the schemes $\cAp$ and $\cXp$ are isomorphic as schemes over $k$, however, these schemes are endowed with different coordinate systems.

We frequently write $N_{\seed}$ (resp. $M_{\seed}$) to stress that $N$ is endowed with the basis $\{e_i\}_{i\in I}$ (resp. $\{e^*_i\}_{i\in I}$). In particular, if $v $ belongs to $ N$ (or to $M$) we write $v_{\seed}$ to denote the column vector expressing $v$ in the corresponding basis. Moreover, given an $n\times n$ matrix $X$ with entries in $\Z$, we write $X:N_{\seed} \to M_{\seed}$ to denote the $\Z$-linear map $N \to M$ induce by $ X$. 

The varieties introduced above are related to each other via distinguished maps, see \cite[\S2.1.3]{BCMNC} and \cite[Appendix A]{GHKK}. Here we recall the description of three of such maps and refer to \emph{loc. cit.} for further details.

The cluster varieties $ \cA$ and $\cX$ are related to each other via a \emph{cluster ensemble map} 
\[
p:\cA \to \cX.
\]
By a slight abuse of notation (that we carry through analogously for the rest of the maps in this subsection and the following one), the restriction of $p$ to the seed torus $\cA_{\seed'}$ is also denoted by $p$. This restriction is a 
 map $p:\cA_{\seed'} \to \cX_{\seed'}$ whose coordinate functions are monomials. Its pull-back is determined by the lattice map $p^*:N \to M$ given by $n_{\seed} \mapsto B_{\seed} n_{\seed}$.
We call $p^*$ a \emph{cluster ensemble lattice map}.

The variety $\cX$ can be described as a geometric quotient of $\cAp$ by the action of a torus canonically isomorphic to $T_N$. We let 
\[
\tilde{p}:\cAp \to \cX.
\]
be the associated quotient map.
Explicitly, the restriction of $\tilde{p} $ to a seed torus $(\cAp)_{\seed'}$ is a map of the form $\tilde{p}:(\cAp)_{\seed'}\to \cX_{\seed'}$ whose coordinates functions are monomials. Its pull-back is determined by the lattice map $\tilde{p}^*:  N \to M\oplus N$ given by $ n \mapsto  (p^*(n),n)$.

There is a distinguished fibration
\[
w:\cXp \to T_M.
\]
The restriction of $w $ to a seed torus $(\cXp)_{\seed'}$ is a map of the form $w:(\cXp)_{\seed'}\to T_M$ whose coordinates functions are monomials. Its pull-back is determined by the lattice map $w^*:N\to N\oplus M$ given by $n \mapsto (n, -p^*(n))$.

Both $\cXp$ and $\cX$ are Poisson varieties and there is a canonical inclusion
\[
\xi:\cA \hookrightarrow \cXp
\]
that realizes $\cA$ as a fiber of $w$. The restriction of $\xi $ to a seed torus $\cA_{\seed'}$ is a map of the form $w:\cA_{\seed'}\to (\cXp)_{\seed'}$ whose coordinates functions are monomials. Its pull-back is determined by the lattice map
$
\xi^*: N\oplus M \to M$ given by $ (n,m) \mapsto p^*(n)-m.$

\subsection{The weight map}

Let $T_L$ be the torus whose cocharacter lattice is $L$. A rational function $f$ on $T_L$ is positive if it can be expressed as a fraction $f=f_1/f_2$, where both $f_1$ and $f_2$ are a linear combination of characters of $T_L$ with coefficients in $\Z_{>0} $.
The positive rational functions on $T_L$ form a semifield inside $ k(T_L)$ denoted by $Q_{\rm sf}(L)$.
A rational map $g:T_L\dashrightarrow T_{L'}$ between two tori is a \emph{positive} if its pullback $g^*:k(T') \to k(T)$ restricts to an isomorphism $g^*:Q_{\rm sf}(L') \to Q_{\rm sf}(L)$.
Given a semifield $P$, we let $T_L(P)$ denote the set $P$ valued points of $T_L$. By definition $T_L(P)$ is the collection of the semifield homomorphisms from $Q_{\rm sf} (L)$ to $ P$ and is written as follows
\[
T_L(P):=\Hom_{\rm sf}(Q_{\rm sf} (L), P).
\]
Notice that the sublattice of monomials of $Q_{\rm sf}(L)$ is canonically identified with $L^*$. 
Considering $P$ just as an abelian group the restriction of an element of $Q_{\rm sf}(L)$ to $L^*$
determines a bijection $T_L(P) \overset{\sim}{\longrightarrow} \Hom_{\rm groups} (L^*, P) $. 
Hence we have a bijection
\begin{equation}
\label{eq:identification}
r:T_L(P)\overset{\sim}{\longrightarrow} L\otimes P
\end{equation}
obtained by composing $T_L(P) \overset{\sim}{\longrightarrow} \Hom_{\rm groups} (L^*, P) $ with the canonical isomorphism $\Hom_{\rm groups}(L^*, P) \cong L \otimes P$. We warn the reader that we use the symbol $r$ for any such identification, even if we vary the lattice $L$ and the semifield $P$.

In this paper we exclusively consider the semifields $\Z^T=(\Z, \max, +)$, $\Z^t=(\Z,\min, +)$ and their real counterparts $\R^T$ and $\R^t$. So from now on $P$ will denote one of these semifields.
We also have a canonical identification $i:T_L(\R^T)\to T_L(\R^t)$ given by a sign change. More precisely, we have a commutative diagram 
\[
\xymatrix{
T_L(\R^T)\ar_{i}[d] \ar^{r}[r] & L_{\R} \ar^{ -1\cdot }[d]\\
T_L(\R^t)  \ar^{r}[r] & L_{\R},
}
\]
where the horizontal arrow on the right is the multiplication by $-1$. 
For this reason, the inverse of $i$ is also denoted by $i$.

The \emph{tropicalization} (with respect to $P$) of a positive rational map $g:T_L\dashrightarrow T_{L'}$ between tori is
\begin{align*}
g(P): T_L(P) & \to T_{L'}(P)\\
 h \ & \mapsto \ h \circ g^*.    
\end{align*}
Since $g(\Z^T)= g(\R^T)\mid_{T_{L}(\Z^T)}$ we denote both $g(\R^T)$ and $g(\Z^T)$ simply by $g^T$.
The notation $g^t$ is defined analogously:
\[
g^T=g(\R^T)\quad \quad \text{and} \quad \quad g^t=g(\R^t).
\]

Moreover, for a function $f:T_{L}(P) \to T_{L'}(P')$ we let
\begin{equation}
\label{eq:ulnotation}
\underline{f}: P\otimes_{\Z}L \to P'\otimes_{\Z}L' 
\end{equation}
be the induced function.
More precisely, $\underline{f}$ is the only function making the following diagram commute
\[
\xymatrix{
T_L(P) \ar^{f}[r] \ar_{r}[d] & T_{L'}(P')\ar^{r}[d]\\
 P\otimes_{\Z}L \ar_{\underline{f}}[r] & P'\otimes_{\Z}L'.
}
\]
Similarly, given a set $S\subset T_L(P)$ we let
\begin{equation}
\label{eq:ulnotationset}
\underline{S}\subset P\otimes_{\Z}L
\end{equation}
be the image of $S$ under the identification $ T_{L}(P)\overset{\sim}{\longrightarrow}P\otimes_{\Z} L $ and for a set $\mathcal{G}\subset  P\otimes_{\Z}L$ we let 
\begin{equation}
\label{eq:ulnotationset2}
\widetilde{\mathcal{G}}\subset T_{L}(P)
\end{equation}
be its inverse image under the identification.

\begin{remark}
In the literature (in particular in \cite{GHKK}) functions of the form $f$ and $\underline{f}$ (and sets of the form $S$ and $\underline{S} $) are usually  denoted by the same symbol and treated indistinguishably, however, for sake of clarity we introduce the aforementioned notation.
\end{remark}

Consider the restriction of $w$ to a seed the torus $(\cXp)_{\seed}$. By a slight abuse of notation we denote this restriction again by $w$, so this is a map of the form $w:(\cXp)_\seed \to T_M$. The \emph{weight map} is the function 
\[
w^T: (\cXp)_{\seed}(\R^T)\to T_M(\R^T).
\]
The reader is referred to \S~\ref{sec:inequalities} for an explanation of this terminology (see in particular the discussion preceding  Lemma~\ref{lem:identification_w_slice}).
Under the canonical identifications of $ (\cXp)_{\seed}(\R^T)$ (resp. $T_M(\R^T)$) with $M_{\R} \oplus N_{\R}$ (resp. $M_{\R}$) the map $\underline{w^T}$ is given as follows:
\begin{align*}
\underline{w^T}&: M_{\R} \oplus N_{\R}\to M_{\R}\\
  & \ \ \ \ \ \ (m,n) \longmapsto m-p^*(n).
\end{align*}

\section{Description of the cluster complex for cluster Poisson varieties}
\label{sec:Desc_cluster_complex}

The aim of this section is to provide a description of the cluster complex associated to a cluster Poisson variety.
Before presenting such description we briefly discuss cluster complexes associated to cluster varieties and the cluster complexes associated to cluster algebras.

The cluster complex $\Delta^+(A)$ associated to a cluster algebra $A$ was introduced in \cite{FZ_II}. 
By definition, its vertices are the cluster variables and its simplices are the sets of compatible cluster variables.
So $\Delta^+(A)$ is independent of a choice of initial seed for $A$.
Moreover, for every $\seed \in \mathbb T_n$, the $\gv$-fan $\Delta^+_{\seed}(A)$ is a realization of $\Delta^+(A)$ as a rational polyhedral fan. We think of the $\Delta^+_{\seed}(A)$ as a fan living in $M_{\R}$.

From the perspective of \cite{GHKK}, every cluster variety $\cV$ endowed with a choice of seed torus $\cV_\seed $ has an associated \emph{Fock--Goncharov cluster complex} $\Delta^+_{\seed}(\cV)$, see Definition~7.9 of \emph{loc. cit.}.
The definition of $\Delta^+_{\seed}(\cV)$ is given using global monomials on $\cV$ (which we discuss in \S~\ref{sec:theta_functions}).
By construction we have that 
\[
\Delta^+_{\seed}(\cA) \subset \cX_\seed(\R^T) \quad \quad \text{and} \quad \quad \Delta^+_{\seed}(\cX) \subset \cA_\seed(\R^t).
\]
Using the identifications of $\cX_\seed(\R^T)$ (resp. $\cA_\seed(\R^t)$) with $M_{\R}$ (resp. $N_{\R}$), it is shown in Lemma 7.10 of \emph{loc. cit.} that $\underline{\Delta^+_{\seed}(\cV)}$ is a rational polyhedral fan and provide a description of $\Delta^+_{\seed}(\cX)$ using $\Delta^+_{\seed}(\cAp)$ and the weight map.
In \S~\ref{sec:inequalities} we recall this construction and describe the cones of $\Delta_{\seed}(\cX)$ using $\cv$-matrices and $\cv$-vectors. More precisely, we explain how every $\cv$-matrix determines a (not necessarily top dimensional) cone of $\Delta^+_{\seed}(\cX)$, and how every $\cv$-vector of a $\cv$-matrix $C$ determines a supporting hyperplane of a cone determined by $C$.

\subsection{Description of the cones and their supporting hyperplanes via {\bf c}-vectors} 
\label{sec:inequalities}

For brevity, in this paper we refer to $\Delta^+_{\seed}(\cV)$ as a $\cV$-cluster complex. 
It is well known that the $\cA$-cluster complex associated to a seed $\seed $ is identified with the fan realization induced by $\seed$ of the cluster complex associated to the corresponding cluster algebra. Namely,
for $\cA=\cA(\seed)$ and $A=A(\seed)$, we have that
\[
\underline{\Delta^+_{\seed}(\cA)} = \Delta^+_{\seed}(A). 
\]
The natural ambient space for $ \Delta^+_{\seed}(\cA)$ is the ``$\max$'' tropical space $\cX_{\seed}(\R^T)$, which is identified with $M_{\R}$. The complex $\Delta^+_{\seed}(\cAp)$ naturally lives inside  the ``$\max$'' tropical space $(\cXp)_{\seed}(\R^T)$, which is identified with $M_\R \oplus N_\R$. Working in the vector space identifications of the tropical and using the notation introduced in \eqref{eq:ulnotationset} spaces we have that 
\[
\underline{\Delta^+_{\seed}(\cAp)}= \underline{\Delta^+_{\seed}(\cA) }\times N_{\R}.
\]
In particular, the maximal cones of $\underline{\Delta^+_{\seed}(\cAp)}$ are not strictly convex in the sense that they contain a linear subspace.
%as it contains the liner space $N_{\R}$ but it is always rational and polyhedral.

As opposed to $\cA$ and $\cAp$, the $\cX$-cluster complex $\Delta^+_{\seed}(\cX)$ naturally lives inside the ``$\min$'' tropical space $\cA_{\seed}(\R^t)$, which is identified with $N_{\R}$.
We briefly explain this difference and refer to \cite[\S3.2.6]{BCMNC} for the relevant details.

There is a natural action of $T_N$ on $\cAp$ and every theta function on $\cAp$ is an eigenfunction with respect to this action, see \cite[Proposition 3.14]{BCMNC}. Since $\cX$ is a quotient of $\cAp$ by the action of $T_N$, the theta functions on $\cX$ are the functions on $\cX$ induced by the theta functions on $\cAp$ whose $T_N$-weight is $0$.
A key observation is that the $T_N$-weight of a theta function on $\cAp$ can be obtained tropicalizing the weight map $w:(\cXp)_\seed \to T_M$ see \cite[Proposition 3.14]{BCMNC} (notice that such proposition is stated in a more general situation; using the notation of \emph{loc. cit.} we are considering the case where $\cA=\cAp$).
In other words, if we consider the \emph{weight zero slice} associated to $\seed$, which is defined as $(w^T)^{-1}(0)\subset (\cXp)_\seed(\R^T)$, then the theta functions on $\cAp$ parametrized by the integral points in $(w^T)^{-1}(0)$ are of $T_N$-weight $0$. The following result was proved in \cite[Lemma 3.15]{BCMNC} in a more general situation.

\begin{lemma}
\label{lem:identification_w_slice}
    For every seed $\seed$ the image of the composition $ \xi^T \circ i:\cA_{\seed}(\R^t)\to (\cXp)_{\seed}(\R^T)$ is precisely $(w^T)^{-1}(0)$.
    Since $\xi^T \circ i$ is injective then 
    $\cA_{\seed}(\R^t)$ and $ (w^T)^{-1}(0)$ are in bijection.
\end{lemma}

We let $j:=\xi^T \circ i$. One way to describe $\Delta^+_\seed(\cX) $ is as the intersection of $\Delta^+_{\seed}(\cAp)$ with the weight zero slice. In light of Lemma~\ref{lem:identification_w_slice}, we can realize this cluster complex inside $\cA_\seed(\R^t)$. So we have the following definitions.

\begin{definition}
    The {\bf cluster complex} for $\cX$ {\bf inside} $\cXp(\R^T)$ is
    \[
    \widetilde{\Delta}^+_\seed(\cX):=\Delta^+_{\seed}(\cAp)\cap (w^{T})^{-1}(0).
    \]
    The {\bf cluster complex} for $\cX$ {\bf inside} $\cA(\R^t)$ is
     \[
    \Delta^+_\seed(\cX):=j^{-1}\lrp{\widetilde{\Delta}^+_\seed(\cX)}
    \]
\end{definition}

Notice that under the vector space identification of $\cA_\seed(\R^t)$ and $\cA_\seed(\R^T)$ (resp. $\cX_\seed(\R^T)$) with $N_{\R}$ (resp. $M_{\R}$), see \eqref{eq:identification}, we have that $p^T\circ i$ is identified with $ p^*$. In other words, using the notation introduced in \eqref{eq:ulnotation} we have that 
\[
\underline{(p^T\circ i)}=p^* 
\]
and the commutative diagram
\[
\xymatrix{
\cA_{\seed}(\R^t) \ar_{r}[d] \ar^{i}[r] & \cA_{\seed}(\R^T) \ar_{r}[d] \ar^{p^T}[r] & \cX_{\seed}(\R^T) \ar^{r}[d]  \\
N_{\R} \ar_{-1}[r] & N_{\R} \ar_{-p^*}[r] & M_{\mathbb R}.
}
\]

Recall from the introduction that we are considering {\bf g}-vectors as elements of $M$, so the {\bf g}-cones $\mathcal{G}_{\seed'}$ belong to $M_{\R}$.

Using the considerations made above, we are ready to describe $ \Delta^+_{\seed}(\cX)$ using the map $p$ and $\Delta^+_{\seed}(\cA)$.

\begin{lemma} \label{cone}
    Every cone of $\underline{\Delta^+_{\seed} (\cX)}$ is of the form
    \[
    (p^*)^{-1}(\mathcal{G}_{\seed'}),
    \]
    where $\mathcal{G}_{\seed'}$ is the {\bf g}-cone in $\underline{\Delta^+_{\seed} (\cA)}$ associated to $\seed'\in \mathbb T_n$. Equivalently, using the identifications $ r $ to endow $\cA_{\seed}(\R^t)$ (resp. $\cX_{\seed}(\R^T)$) with the linear structure of $N_{\R}$ (resp. $M_{\R}$), we have that every cone of $\Delta^+_{\seed} (\cX)$ is of the form
    \[
    (p^T\circ i)^{-1}(\widetilde{\mathcal{G}_{\seed'}}).
    \]
\end{lemma}
\begin{proof}
    It is enough to prove the first part of the statement. We begin by noticing that every cone of $\underline{\Delta^+_\seed(\cX)} $ can be described as the intersection of a cone of $ \underline{\Delta^+_{\seed}(\cAp)}$ with $\underline{(w^T)}^{-1}(0)$. 
    Moreover, every cone of $ \underline{\Delta^+_{\seed}(\cAp)}$ is of the form $\mathcal{G}_{\seed'} \times N_{\R} $ for some $\seed' \in \mathbb T_n$.
    Since $\underline{(w^T)}^{-1}(0)= \{ (p^*(n),n)\mid n\in N_{\R}\}$, we have that
    \[
    \mathcal{G}_{\seed'} \times N_{\R} \cap \underline{(w^T)}^{-1}(0) = \{ p^*(n),n) \mid p^*(n) \in \mathcal{G}_{\seed'}\}.
    %\mathcal{G}_{\seed'} \times N_{\R} \cap (w^T)^{-1}(0) = \{ (p^*(n),n) \mid p^*(n) \in \mathcal{G}_{\seed'}\}.
    \]
    Finally, the map $\underline{j}:N_{\R}\to M_{\R}\oplus N_{\R}$ is given by
    \[
    %j(n)=(p^*(n),n)\in M_{\R} \oplus N_{\R},
    \underline{j}(n)=(p^*(n),n)\in M_{\R} \oplus N_{\R},
    \]
    see \cite[Lemma 3.15]{BCMNC}. The result follows.
\end{proof}

Let us now recall the tropical duality between $\cv$- and $\gv$- matrices introduced in \cite{NZ}. Let $G^{B_{\seed}}_{\seed'}$ (resp. $C^{-B_{\seed}^T}_{\seed'}$) be the $\gv$-matrix (resp. $\cv$-matrix) associated to $\seed'\in \mathbb T_n$ in the cluster pattern with initial matrix $B_{\seed}$ (resp. $-B_{\seed}^T$).
Then it was proved in \emph{loc. cit.} that $G^{B_{\seed}}_{\seed'}$ is invertible over $ \Z$ and that
\[
(G^{B_{\seed}}_{\seed'})^{-1}=(C^{-B_{\seed}^T}_{\seed'})^T.
\]
Combining Lemma~\ref{cone} and tropical duality one obtains that every $\cv$-vector ($\cv$-matrix) determines a supporting hyperplane (the system of inequalities) of a cone of $\underline{\Delta^+_{\seed}(\cX)}$.

\begin{corollary}
\label{conedescription}
Let $\seed,\seed'\in \mathbb T_n$ and consider the $\gv$-cone $\mathcal{G}_{\seed'}$ of $\underline{\Delta^+_{\seed}(\cA)}$. Then the cone $(p^*)^{-1}(\mathcal G_{\seed'})$ of $\underline{\Delta^+_{\seed}(\cX)}$ is given as
\[
(p^*)^{-1}(\mathcal G_{\seed'})=\{\beta \in N_\R \; | \; (C_{\seed'}^{-B_{\seed}^T})^TB_{\seed}\beta_{\seed} \geq \textbf{0}\}.
\]
\end{corollary}

\begin{proof}
Let $\beta^T_{\seed}=(\beta_1,\dots, \beta_n)$ be such that 
\[
p^{*}(\beta_{\seed})=\sum_{i=1}^{n}\alpha_{i}\gv_{i;\seed'}, \quad \text{where } \alpha_{1}, \dots, \alpha_{n}\geq 0. 
\]
This equation can be expressed as $p^{*}(\beta_{\seed})= G^{B_{\seed}}_{\seed'} \alpha$, where $\alpha^{T}=(\alpha_{1}, \dots, \alpha_{n})$.  
Using tropical duality we obtain that 
\[
(C_{\seed'}^{-B_{\seed}^T})^T B_{\seed}\beta_{\seed}= \alpha \geq \textbf{0}.
\]    
\end{proof}
%In particular, we find that $\alpha_{i;\seed}=(c_{i;\seed}^{-B^T})^T p^{*}(\beta)$, implying $(c_{i;\seed}^{-B^T})^T p^{*}(\beta)\geq 0$. In other words, the cone $(p^*)^{-1}(\mathcal G_{\seed})$ is defined by the inequalities:
In particular, if $(c_{1;\seed'},\dots, c_{n;\seed'})$ are the columns of (\ie the $\cv$-vectors that form) the $\cv$-matrix $C_{\seed'}^{-B_{\seed}^T} $ then the inequalities defining the cone $(p^*)^{-1}(\mathcal G_{\seed'})$ are
\begin{eqnarray*}
(c_{1;\seed'}^{-B_{\seed}^T})^TB_{\seed} \beta_{\seed} \geq 0\phantom{.}\\
\vdots  \phantom{aaaaaa}\\
(c_{n;\seed'}^{-B_{\seed}^T})^T B_{\seed} \beta_{\seed} \geq 0.
\end{eqnarray*}
In other words, every $\cv$-vector $c_{j;\seed'}$ defines a supporting hyperplane of the cone $(p^*)^{-1}(\mathcal G_{\seed'})$.

\begin{remark}
\label{rem:frozen_directions}
Let $A$ (resp. $A'$) be the cluster algebra associated to a seed $\seed$ (resp. $\seed'$), where $\seed'$ is obtained from $\seed$ by adding frozen directions (that is we enlarge $I$ by letting $I\setminus \Iuf$ grow and extend the bilinear form of $\seed$ arbitrarily).  
It is a deep result (that follows for example from \cite{GHKK}) that $ \Delta(A)$ and $\Delta(A')$ are isomorphic.
However, at the level of the fan realizations there is a small change. Namely, if $\cA$ (resp. $\cA'$) is the cluster $\cA$-variety associated to $\seed$ (resp. $\seed'$) then $\underline{\Delta^+_{\seed'}(\cA')}$ is isomorphic to the product of $\underline{\Delta^+_{\seed}(\cA)}$ with a real vector space whose dimension equals the number of frozen indices that have been added.
In contrast, if we let $ \cX$ (resp. $\cX'$) be the cluster $\cX$-variety associated to $\seed$ (resp. $\seed'$) then $\underline{\Delta^+_{\seed'}(\cX')}$ in general \emph{is not} the product of $\underline{\Delta^+_{\seed}(\cX)}$ with a vector space.
Indeed, we can see from Lemma~\ref{cone} that the vector space spanned by the kernel of $B_{\seed}$ belongs to every cone of $\underline{\Delta^+_{\seed}(\cX)}$.
So if the rank of $B_{\seed'}$ is larger than the rank of $B_{\seed}$ then $\underline{\Delta^+_{\seed'}(\cX')}$ might not be isomorphic to a product of $\underline{\Delta^+_{\seed}(\cX)}$ with a vector space.
For instance, suppose that $\cX$ is associated to an orientation of Dynkin diagram of type $\mathbb A_3$. Then $\underline{\Delta_{\seed}^+(\cX)}$ has $6$ maximal cones each of which has a one dimensional linear subspace, see Example~\ref{A3} for a more precise description.
Let $\cX'$ be obtained by adding a frozen vertex so that the support of the second quiver is of type $\mathbb A_4$. 
Here the corresponding exchange matrix has full rank and $\underline{\Delta^+_{\seed'}(\cX')}$ has 9 maximal cones all of which are strictly convex.
In fact in this case we have that $\underline{\Delta^+_{\seed'}(\cX')}$ and $\underline{\Delta^+_{\seed'}(\cA')}$ are isomorphic.
\end{remark}

\begin{remark}
    In \cite{Boss}, Bossinger constructs a class of fans associated to certain partial compactifications of $\cA$. More precisely, the fans are the totally positive part of the tropicalization of the ideal defining the compactification (see \cite[Theorem 1.1]{Boss}). In the cases considered in \emph{loc. cit.}, some cones of $\underline{\Delta^+_{\seed}(\cX)}$ can be obtained as projections of cones in the totally positive part of the tropicalization. This provides a way to describe the fan $\underline{\Delta^+_{\seed}(\cX)}$ locally using Gr\"obner theory. 
\end{remark}

\subsection{Dimension and facets of the cones}
\label{dim}
It is possible to compute the dimension of the cone $(p^{*})^{-1}(\mathcal{G}_{\seed'})$ using its description by inequalities. The most interesting case is when $p^{*}$ is not an isomorphism (in case $p^*$ is an isomorphism then $(p^{*})^{-1}(\mathcal{G}_{\seed'})$ is always a top dimensional cone). 
For brevity, in this subsection we write $ C_{\seed'}^T$ as a shorthand for $(C_{\seed'}^{-B_{\seed}^T})^T$.
We consider $C_{\seed'}^TB_{\seed}\beta_{\seed} \geq \textbf{0}$ as a system of inequalities on variables $ \beta_1, \dots, \beta_n$.
Following \cite{CCZ}, we say that the inequality $c_{k;\seed'}^T B_{\seed} \beta_{\seed} \geq 0$ is an \textit{implicit equality} of the system $C_{\seed'}^T B_{\seed}\beta_{\seed} \geq \textbf{0}$ if whenever $b^T=(b_1, \dots, b_n)\in \R^n$ is such that $c_{i;\seed'}^T B_{\seed} b \geq 0$ for all $i\in \{1, \dots, n\}$ we have that $c_{k;\seed'}^T B_{\seed} b = 0$. The collection of the implicit equalities of the system is denoted as $[C_{\seed'}^T B_{\seed}]^{=}\beta_{\seed} \geq \textbf{0}^{=}$ and the collection of the remaining inequalities as $[C_{\seed'}^TB_{\seed}]^{>}\beta_{\seed} \geq \textbf{0}^{>}$.
In particular, we can think of both $[C_{\seed'}^T B_{\seed}]^{=}\beta_{\seed} \geq \textbf{0}^{=}$ and $[C_{\seed'}^TB_{\seed}]^{>}\beta_{\seed} \geq \textbf{0}^{>}$ as systems of inequalities (with potentially fewer inequalities than the system $C_{\seed'}^TB_{\seed} \beta_{\seed} \geq \textbf{0}$).

Since $(p^*)^{-1}(\mathcal G_{\seed'})\neq \emptyset$ then \cite[Theorem 3.17]{CCZ} tells us that
\begin{equation}
\label{eq:dim}
\dim((p^*)^{-1}(\mathcal G_{\seed'}))+\rank([C_{\seed'}^TB_{\seed}]^=)=n,
\end{equation}
where $[C_{\seed'}^TB_\seed]^=$ represents the matrix associated with the system of implicit equalities $[C_{\seed'}^TB_\seed]^{=}\beta_{\seed} \geq \textbf{0}^{=}$. 
More precisely, if we let $J^{=}\subseteq \{1,\dots,n\}$ be the set of indices corresponding to the implicit equalities in the system $[C_{\seed'}^TB_\seed]^{=}\beta_{\seed} \geq \textbf{0}^{=}$, then the rank of the matrix with rows $\{c_{j;\seed'}^T B_\seed\}_{j\in J^{=}}$ allows us to determine the dimension of $(p^{*})^{-1}(\mathcal G_{\seed'})$. In particular, if $C_{\seed'}^TB_\seed  \beta_\seed \geq \textbf{0}$ has no implicit equalities, except possibly for the trivial equality 0 = 0, then $(p^{*})^{-1}(\mathcal G_{\seed'})$ has full dimension $n$.

To characterize the implicit equalities of the system $C_{\seed'}^T B_{\seed}\beta_{\seed} \geq \textbf{0}$ notice that if there is a subset $\{c_{h;\seed'}^{T} B_{\seed} \beta_{\seed}  \geq 0\}_{h\in H}$ of inequalities in the system parametrized by some nonempty set $H\subseteq \{1,\dots,n\}$ and there are positive integers $\lambda_h> 0$, for all $h\in H$, such that $\displaystyle{\sum_{h\in H}\lambda_h (c_{h;\seed'}^T B_{\seed}  \beta_{\seed})=0}$ then
\[
-c_{k;\seed'}^T  B_{\seed}\beta_{\seed} = \frac{1}{\lambda_k} \sum_{\substack{h\in H \\ h\neq k}}\lambda_h (c_{h;\seed'}^T  B_{\seed} \beta_{\seed})\geq 0,
\] 
for all $k\in H$. This implies that $c_{k;\seed'}^T  B_{\seed} b=0$ for every solution $b\in \R^n$ of the system $C_{\seed'}^TB_{\seed} \beta_{\seed} \geq \textbf{0}$. So in this situation we conclude that $c_{k;\seed'}^T  B_{\seed} \beta_\seed \geq  0$ is an implicit equality of the system.

In particular, a set of inequalities $\{c_{h;\seed'}^{T}  B_{\seed} \beta_{\seed}  \geq 0\}_{h\in H}$ is contained in the set of implicit equalities if and only if there are positive integers  $\{\lambda_h\}_{h\in H}$ such that $\sum_{h\in H}\lambda_hc_{h;\seed'} \in \ker(p^{*})$.
%In particular, one can search for positive linear combinations of the {\bf c}-vectors $c^T_{1;\seed'}, \dots c^T_{n;\seed'}$ that lie in the kernel of $p^*$ in order to determine the implicit equalities of the system 
In the skew-symmetric acyclic case inequalities belonging to the set of implicit equalities can be determined by finding positive linear combinations of $\gv$-vectors, see Remark~\ref{rem:relationgvectors}.  
In particular, the observations made above can be used to identify the implicit equalities of the system $C_{\seed'}^T B_{\seed} \beta_{\seed} \geq \textbf{0}$.

We now give a bound on the number of facets of $(p^*)^{-1}(\mathcal G_{\seed'})$. For this, we let $J^{>}$ be the set of indices corresponding to the inequalities in the system $[C_{\seed'}^TB_{\seed}]^{>}\beta_{\seed} \geq \textbf{0}^{>}$. Then 
\[
(p^*)^{-1}(\mathcal G_{\seed'})=\left\{\beta\in N_\R \; | \; c_{i;\seed'}^{T} B_{\seed} \beta_{\seed} = 0, i\in J^{=}, \; c_{j;\seed'}^{T}  B_{\seed} \beta_{\seed} \geq 0, j\in J^{>}\right\}.
\]
We say that the system $[C_{\seed'}^TB_{\seed}]^{>}\beta_{\seed} \geq \textbf{0}^{>}$ is \emph{redundant} if there is a $k\in J^{>}$ such that
\[
(p^*)^{-1}(\mathcal G_{\seed'})=\left\{\beta\in N_\R \; |  c_{i;\seed'}^{T}  B_{\seed} \beta_{\seed} = 0, i\in J^{=}\text{ and }c_{j;\seed'}^{T} B_{\seed} \beta_{\seed} \geq 0, j\in J^{>}\setminus \{k\}\right\}.
\]
In this case we say that the linear form $c_{k;\seed'}^{T}  B_{\seed} \beta_{\seed}$ is superfluous for the system.
Observe, moreover, that if there exist $\lambda_1, \dots ,\lambda_p> 0$ such that $c_{k;\seed'}^{T}  B_{\seed} \beta_{\seed}=\sum_{\substack{i=1\\i\neq k}}^{p} \lambda_{i} c_{i;\seed'}^{T}  B_{\seed} \beta_{\seed}$, for $k\in J^{>}$, then the inequality $c_{k;\seed'}^{T}  B_{\seed} \beta_{\seed}$ is superfluous.
Then, following \cite[Theorem 3.27]{CCZ} for each facet $F$ of $(p^{*})^{-1}(\mathcal G_{\seed'})$, there exists $j \in J^{>}$ such that the inequality $c_{j;\seed'}^{T}  B_{\seed} \beta_{\seed} \geq 0$ defines $F$, that is, 
\[F := (p^{*})^{-1}(\mathcal G_{\seed'}) \cap \{\beta \in N_{\R} \mid c_{j;\seed'}^{T}  B_{\seed}  \beta_{\seed} = 0\}.\]
This implies that the number of facets of $(p^{*})^{-1}(\mathcal G_{\seed'})$ is less than or equal to $|J^{>}|=n-|J^{=}|$ and that the equality is achieved only when the system $[C_{\seed'}^T B_{\seed}]^{>} \beta_{\seed} \geq \textbf{0}^{>}$ is not redundant. 

\subsection{The skew-symmetrizable case}
\label{skew-symmetrizable}

For simplicity we have considered the skew-symmetric case, that is, we have assumed that the codomain of the bilinear form $\{\cdot , \cdot \}$ is $\Z$.
However, the results of \S~\ref{sec:inequalities} and \S~\ref{dim} still hold (upon a suitable reinterpretation of some constructions) in the skew-symmetrizable case as we proceed to explain.

Suppose we are in the skew-symmetrizable case. 
This means that the codomain of the skew-symmetric bilinear form $\{\cdot, \cdot \}$ is $\Q$ and that there is a finite index sublattice $N^{\circ}\subset N $ such that $ \{ N, N^\circ \}\subset \Z$. Let $M^{\circ}\supset M$ be the $\Z$-dual of $N$. In this case the seed $\seed=(e_i)_{i\in I}$ is such that $\{e_i\}_{i\in I}$ is a basis of $N$ and $\{D_ie_i\}_{i\in I}$ is a basis of $N^{\circ}$ for positive integers $D_1, \dots, D_n\in \Z_{>0}$. 
Moreover, $\{e_i^*\}_{i\in I}$ is the basis for $M$ dual to $\{e_i\}_{i\in I}$ and $\{f_i\}_{i\in I}$ is a basis of $M^\circ$, where $f_i=D_i^{-1}e_i^*$.
Let $D=\text{lcm}(D_i\mid i\in I)$.
In this case the $ij$ entry of the matrix $B_{\seed}$ is $\{e_j,e_i\}D_i$, the corresponding cluster ensemble lattice map is of the form $p^*: N \to  M^{\circ}$ and its matrix in the bases of $N$ and $M^{\circ}$ described above is still $B_{\seed}$.

In this case the $\gv$-vectors belong to $M^\circ$ and the $\cv$-vectors still belong to $N$, see Remark~\ref{d_factor}. 
The ambient vector space for $\underline{\Delta^+_{\seed}(\cX)}$ is $N_{\R}$ and for $\underline{\Delta^+_{\seed}(\cA)}$ is $M^{\circ}_{\R}$.
The first assertion of Lemma~\ref{cone} still holds as stated in the skew-symmetrizable case and its proof is essentially the same. 
Indeed, the main difference in the proof is that one would need to discuss the weight map in the skew-symmetrizable case; the reader can find the relevant discussion in Proposition 3.14 of \cite{BCMNC}. 
Moreover, the corresponding weight $0$ slice is a subset of $M^\circ_{\R} \oplus N_{\R}$ and, just as in the skew-symmetric case, it coincides with the image of the inclusion $N_{\R}\to (M^\circ \oplus N)_{\R} $ given by $n \mapsto (p^*(n),n) $, see \S3.2.6 of \cite{BCMNC} for the explanation. So under these considerations, the scheme of the proof remains the same.

However, in order to state in the skew-symmetrizable case the second assertion of Lemma~\ref{cone} one needs to talk about \emph{Fock-Goncharov duality} for cluster varieties.
We proceed to elaborate on this and refer to \S3.1 of \cite{BCMNC} for the details we shall omit.
Denote by $\Gamma$ the fixed data giving rise to $\cA$ and $\cX$, and in particular to the cluster ensemble map $p:\cA \to \cX$. We can consider the \emph{Langlands dual} fixed data $\Gamma^\vee$. The Fock--Goncharov dual of $\cA$ (resp. $\cX$) is the cluster $\cX$-variety (resp. $\cA$-variety) associated to $\Gamma^\vee$, which we denote by $\cA^\vee$ (resp. $\cX^\vee$).
The seed tori for $\cA^\vee$ (resp. $\cX^\vee$) are of the form $T_{M^\circ}$ (resp. $T_{D\cdot N}$).
The corresponding cluster ensemble map is $p^\vee: \cX^\vee \to \cA^\vee$ and satisfies that $\underline{((p^\vee)^T\circ i)}=p^*$. In other words, under the canonical isomorphism of $(DN)_{\R} \cong N_\R$, we have that
\[
\underline{((p^\vee)^T\circ i)}: N_{\R} \to M^\circ_{\R} \quad \text{is given by} \quad n \mapsto p^*(n).
\]
Next, for every $\seed\in \T_n $ for $\Gamma $ we have that
\[
\Delta^+_{\seed}(\cA) \subset \cA^{\vee}_{\seed}(\R^T) \quad \quad \text{and} \quad \quad \Delta^+_{\seed}(\cX) \subset \cX^\vee_{\seed}(\R^t).
\]
In light of the above considerations we have that in order to state the second assertion of Lemma~\ref{cone} in the skew-symmetrizable case it is enough to replace 
$p^T\circ i $ by $(p^\vee)^{\T}\circ i$. Observe that in the skew-symmetric case we have that $p^\vee=p$.

Finally, Corollary~\ref{conedescription} still holds as stated in the skew-symmetrizable case. Indeed, this result follows directly from Lemma~\ref{cone} and tropical duality. The way we have presented tropical duality includes the skew-symmetrizable case and does not need to be modified at all.

\begin{remark}
\label{d_factor}
In some constructions the lattice $D\cdot N$ arises more naturally as the ambient lattice for the {\bf c}-vectors, see for example \S3.2.6 of \cite{BCMNC}. However, we can systematically consider the canonical isomorphism between $ D\cdot N$ and $N$ to eliminate the factor of $D$.
\end{remark}

\section{Global monomials on cluster Poisson varieties} \label{sec:theta_functions}

A \emph{global monomial} on a cluster variety $\cV$ is a regular function on $\cV$ which restricts to a character of some seed torus in the atlas of $\cV$ (see \cite[Definition 0.1]{GHKK}).
The integral points of a cone of the cluster complex $\Delta^+_{\seed}(\cV)$ parametrize those global monomials on $\cV$ that restrict to a character of the same seed torus and, moreover, every such global monomial is a \emph{theta function} on $\cV$.
More precisely, if $v_1, \dots, v_r$ are primitive ray generators of a cone of $\Delta^+_{\seed}(\cV) $ then there are theta functions $\tf^{\cV}_{v_1}, \dots, \tf^{\cV}_{v_r}$ on $\cV$ and a seed $\seed'$ such that every $\tf^{\cV}_{v_j}$ restricts to a character of $\cV_{\seed'}$. Moreover, given non-negative integers $a_1,\dots , a_r$ and $v=\sum_{j=1}^ra_jv_j$ then the theta function $\tf^{\cV}_{v}$ on $\cV$ is such that
\[
\tf^{\cV}_{v}= \prod_{j=1}^r (\tf^{\cV}_{v_j})^{a_j}.
\]

By \cite[Lemma 7.8]{GHKK} a global monomial on $\cA$ is the same as a cluster monomial. 
Every global monomial on $\cX$ can be pulled-back to $\cAp$ along the map $\tilde{p}:\cAp \to \cX$ obtained composing the quotient map $\cAp \to  \cAp/T_N$ with the canonical isomorphism $\cAp /T_N \to \cX$ (see \cite[Lemma 7.10 (3)]{GHKK} for a proof a this statement and \cite[\S2.1.3.]{BCMNC} for a precise description of the maps).
This allows to describe the global monomials on $\cX$ using the \emph{$F$-polynomials} introduced in \cite[Definition 3.3]{FZ_IV} (see Lemma~\ref{initiallemma} bellow). 

Let $y_1,\dots, y_n$ be the cluster coordinates of the seed torus  $\cX_{\seed}$ and $x_1,\dots , x_n,t_1,\dots, t_n$ be the coordinated of the seed torus $(\cAp)_{\seed}=T_M\times T_N$. We use the standard vector notation for a monomial on variables $y_1,\dots, y_n$. Namely, if $v=(v_1,\dots, v_n)\in \Z^n$ then ${\bf y}^v=y_1^{v_1}\cdots y_n^{v_n}$ and similar notation is used for the $x_i's$ and the $t_i's$. 
The $F$-polynomial associated to $B_{\seed}$, $\seed'$ and $i$ is denoted by $F_{i;\seed'}$.

\begin{lemma} \label{initiallemma}
Let $\beta^T=(\beta_1,\dots, \beta_n)\in N$ be such that $p^{*}(\beta)\in \mathcal G_{\seed'}$ for some ${\seed'}\in \mathbb{T}_n$.  Let 
\[
p^{*}(\beta)=\sum_{i=1}^{n}\alpha_{i}\gv_{i;\seed'} \quad \text{where } \alpha_{1}, \dots, \alpha_{n}\geq 0. %\alpha_{i;\seed}\geq 0 \quad \text{for all} \quad i\in[n].
\]
Then
\begin{equation}
\label{eq:gm1}
\vartheta_{\beta}^{\mathcal{X}}:={\bf y}^{\beta_{\seed}}\prod_{i=1}^{n} F_{i;\seed'}(y_1, \ldots, y_n)^{\alpha_{i}}.
\end{equation}
In particular,
\begin{equation}
\label{eq:gm2}
\tf^{\cX}_{\beta}={\bf y}^{\beta_{\seed}}\prod_{i=1}^n F_{i;\seed'}(y_1,\dots,y_n)^{c^T_{i;\seed'}  B_{\seed} \beta_{\seed}}.
\end{equation}
\end{lemma}

\begin{proof}
The pull-back $\tilde{p}^*: N \to M\oplus N$ is given by $\beta \mapsto (p^*(\beta),\beta)$. 
Using the fact that $\tf^{\cAp}_{(p^*(\beta),0)}$ is the global monomial parametrized by $ \sum_{i=1}^{n}\alpha_{i}\gv_{i;\seed'}$ we have that
\begin{equation}
\label{eq:tf1}
\tilde{p}^*(\tf^\cX_\beta)=\tf^{\cAp}_{(p^*(\beta),\beta)}=\tf^{\cAp}_{(0,\beta)}\tf^{\cAp}_{(p^*(\beta),0)}={\bf t}^{\beta_{\seed}}\prod_{i=1}^n\left(\tf^{\cAp}_{(\gv_{i;\seed'},0)}\right)^{\alpha_{i}}.
\end{equation}
The theta function $\tf^{\cAp}_{(\gv_{i;\seed'},0)}$ is in fact a cluster variable so it is given as
\[
\tf^{\cAp}_{(\gv_{i;\seed'},0)}={\bf x}^{\gv_{i;\seed'}}F_{i;\seed'}(\hat{y}_1, \dots, \hat{y}_n)
\]
where 
\begin{equation}
\label{eq:tf2}
\hat{y}_j=t_j\prod_{i=1}^n x_i^{b_{ij}}=\tf^{\cAp}_{(0,e_j)}\tf^{\cAp}_{(p^*(e_j),0)}= \tf^{\cAp}_{(p^*(e_j),e_j)}= \tilde{p}^*(y_j),
\end{equation}
for $j \in \{1,\dots ,n\}$ is the weight 0 monomial on $(\cAp)_{\seed}$ inducing $y_j$.
More precisely, on the one hand $\hat{y}_j $ is a global function on the torus $(\cAp)_{\seed}$ whose weight is $0$ under the action of $T_N$ on $\cAp$, on the other hand the isomorphism $\cAp/T_N \to \cX$ restricts to an isomorphism of torus $(\cAp)_{\seed}/T_N\overset{\sim}{\to} \cX_\seed$. Since $\hat{y}_j $ has $T_N$-weight $0 $ it induces a function on $(\cAp)_{\seed}/T_N$, such induced function corresponds to $y_j$ under the isomorphism $(\cAp)_{\seed}/T_N\overset{\sim}{\to}  \cX_{\seed}$.
Putting together equations \eqref{eq:tf1} and \eqref{eq:tf2} we obtain
\begin{eqnarray*}
\tilde{p}^*(\tf^\cX_\beta)&=&{\bf t}^{\beta_{\seed}}\prod_{i=1}^{n}({\bf x}^{\alpha_{i}\gv_{i;\seed'}})\prod_{i=1}^{n}F_{i;\seed'}(\tilde{p}^*(y_1),\dots, \tilde{p}^*(y_n))^{\alpha_i}\\
&=&{\bf t}^{\beta_\seed}{\bf x}^{p^*(\beta_\seed)}\prod_{i=1}^{n}F_{i;\seed'}(\tilde{p}^*(y_1),\dots, \tilde{p}^*(y_n))^{\alpha_i}\\
&=&\tilde{p}^*({\bf y}^{\beta_\seed})\prod_{i=1}^{n}\tilde{p}^*(F_{i;\seed'}(y_1,\dots, y_n)^{\alpha_i})\\
&=&\tilde{p}^*\left({\bf y}^{\beta_\seed}\prod_{i=1}^{n} F_{i;\seed'}(y_1, \ldots, y_n)^{\alpha_{i}}\right).
\end{eqnarray*}
Equation \eqref{eq:gm1} follows from the injectivity of $\tilde{p}^* $ and \eqref{eq:gm2} from the fact that $ \alpha_{i}=c^T_{i;\seed'} B_\seed   \beta_\seed$.
\end{proof}

\begin{example}
Let $B_{\seed}=\begin{psmallmatrix}
0 & 1 & 0 \\
-1 & 0 & -1 \\
0 & 1 & 0 
\end{psmallmatrix}$ and $\beta^{T}_\seed=(\beta_1,\beta_2,\beta_3)$ be such that $p^{*}(\beta_{\seed})\in \mathcal G_{\seed'}$, where $G^B_{\seed'}=\begin{psmallmatrix}
-1 & 0 & 0 \\
1 & 1 & 1 \\
0 & 0 & 1 
\end{psmallmatrix}$ is its associated $\gv$-matrix. Hence, 

\begin{align*}
\vartheta_{\beta_{\seed}}^{\mathcal{X}}&=y_1^{\beta_1} y_2^{\beta_2} y_3^{\beta_3}(1+y_1)^{(-1,0,0) B_{\seed} \beta_{\seed}}1^{(1,1,1)  B_{\seed} \beta_{\seed}}(1+y_3)^{(0,0,-1)  B_{\seed} \beta_{\seed}}  \\
&= y_1^{\beta_1} y_2^{\beta_2} y_3^{\beta_3}(1+y_1)^{(0,-1,0) \beta_{\seed}}(1+y_3)^{(0,-1,0) \beta_{\seed}}\\
&= y_1^{\beta_1} y_2^{\beta_2} y_3^{\beta_3}(1+y_1)^{-\beta_2}(1+y_3)^{-\beta_2},
\end{align*}
where $-\beta_2\geq 0$ and $-\beta_1+2\beta_2-\beta_3\geq 0$ are the inequalities in the system $C_{\seed'}^TB_\seed  \beta_{\seed} \geq 0$, $C_{\seed'}=\begin{psmallmatrix}
-1 & 1 & 0 \\
0 & 1 & 0 \\
0 & 1 & -1 
\end{psmallmatrix}$.
\end{example}

\begin{remark}
\label{rem:tf_sk}
    Lemma~\ref{initiallemma} also holds as stated in the skew-symmetrizable case (see Remark~\ref{d_factor}). The only difference in the preamble is that the torus acting on $\cAp$ is $T_{N^\circ}$ as opposed to $T_N$.
\end{remark}

\section{Normal vectors via {\bf g}-vectors in the acyclic case} \label{sec:supporting_hyperplanes}
\subsection{The cluster ensemble map via representation theory}
From now on we assume that $Q=Q_\seed$ is acyclic and denote $B_\seed$ also by $B_Q$.
In this section we describe the map $p^*$ using representation theory and use this to provide a formula (see Theorem~\ref{dv-gv} below) that allows to compute the image under $p^*$ of a dimension vector using Auslander-Reiten theory.
We begin by recalling the notions of dimension vectors for objects of $D^b_{kQ}$ and {\bf g}-vectors for objects of $ \KKb(\proj kQ)$.
The reader is referred to \cite{AIR} and \cite{DIJ} for a complete treatment on {\bf g}-vectors using silting theory. 

\subsubsection{Dimension vectors and {\bf g}-vectors} 
All $kQ$-modules are thought of as stalk complexes concentrated in degree $0$.
Let $ S_1, \dots , S_n$ be the simple $kQ$-modules and $P_1, \dots, P_n$ be the indecomposable projective $kQ$-modules.
The classes of the simple modules form a basis of $K_0(D^b_{kQ})$ and the classes of the indecomposable projective modules form a basis of $K_0(\KKb(\proj kQ))$.

We shall use the following conventions. The dimension vector of an object $X$ of $ D^b_{kQ}$ is
\[
\udim \ X = \sum_{i=1}^n d_i e_i\in N,
\]
 where
\[
[X]=\sum_{i=1}^n d_i [S_i].
\]
Let $P=P_1\oplus \dots \oplus P_n $.
The {\bf g}-vector with respect to $P$ of an object $X$ of $ \KKb(\proj kQ)$ is 
\begin{equation*}
    \gv_P^X:= \sum_{i=1}^n g_i e^*_i \in  M, 
\end{equation*}
where
\[
[X]= \sum_{i=1}^n g_i[P_i] \in K_0(\KKb(\proj kQ)). 
\]
Similarly, let $I_1, \dots , I_n$ be the indecomposable injective $kQ$-modules and let $K^b(\inj kQ)$ be the homotopy category of bounded complexes of finitely generated injective $kQ$-modules. The classes of the indecomposable injective modules form a basis of the Grothendieck ring $K_0(\inj kQ)$. In particular we can define
\begin{equation*}
    \gv_I^X:= \sum_{i=1}^n g_i e^*_i \in  M, 
\end{equation*}
where
\[
[X]= \sum_{i=1}^n g_i[I_i] \in K_0(\KKb(\inj kQ)). 
\]

\begin{remark}
    It is well known that if $X\in \KKb(\proj kQ)$ is the projective presentation of an indecomposable rigid module then $\gv_P^X$ corresponds to the {\bf g}-vector of a cluster variable. 
 \end{remark}

\begin{remark}
\label{rem:additivity}
  Since $\text{mod } kQ $ is hereditary the {\bf g}-vectors are additive in exact triangles. %That is if $0\to X \to E \to Y \to 0$ is a short exact sequence then $\gv_P^{P_E}=\gv_P^{P_X} + \gv_P^{P_Y} $.
  In other words, if $X \to E \to Y \to X[1]$ is an exact triangle in $\KKb(\proj kQ)$ (resp. in $\KKb(\inj kQ)$) then $\gv_P^{E}=\gv_P^{X} + \gv_P^{Y} $ (resp. $\gv_I^{E}=\gv_I^{X} + \gv_I^{Y} $).
\end{remark}

\subsubsection{Interpretation of $p^*$ via quiver representations}
Let $C_{kQ}$ be the Cartan matrix associated to $kQ$.
This is a $n\times n$ integer matrix whose columns are the dimension vectors of the indecomposable projective $kQ$-modules.
%; the columns of the transposed matrix $C^{T}_{kQ}$ are the dimension vectors of the indecomposable injective $kQ$-modules.
Since $Q$ is acyclic, the \emph{Euler characteristic} of $kQ$ is the bilinear map $ \langle -, - \rangle_{kQ}: N\times N \to \Z$ given by 
\[
\langle \udim \ X, \udim \ Y \rangle_{kQ} = \dim_k \Hom_{kQ} (X,Y)- \dim_k \Ext^1_{kQ} (X,Y).
\]
Moreover, in the basis of $N$ given by $ \seed$ (\ie by the dimension vectors of the simple modules) the matrix associated to $\langle -, - \rangle_{kQ}$ is $(C^{-1}_{kQ})^T$.

The following is a well known result that asserts that the exchange matrix $B_Q$ can be described using the Cartan matrix. For the convenience of the reader we include a proof. To keep notation light we denote $ (\udim\ X)_{\seed}$ simply by $\udim\ X$.

\begin{lemma}\label{B-C} 
Let $Q$ be an acyclic quiver. Then \[B_Q=(C_{kQ}^{-1})^T-C_{kQ}^{-1}.\]
\end{lemma}

\begin{proof}
Let $c_{ij}$ be the entry in the $i^{\rm th}$ row and the $j^{\rm th}$ column of the matrix $(C_{kQ}^{-1})^T-C_{kQ}^{-1}$, then
\begin{align*}
c_{ij}= &  \ (\udim \; S_i )  (C_{kQ}^{-1})^T  (\udim \; S_j )^T - (\udim \; S_i)   (C^{-1}_{kQ}) (\udim \; S_j )^T \\
= & \ (\udim \; S_i) (C_{kQ}^{-1})^T (\udim \; S_j )^T - (\udim \; S_i) (C^{-1}_{kQ^{op}})^T (\udim \; S_j )^T\\
= &\ \langle \udim \ S_i, \udim  \ S_j \rangle_{kQ}- \langle \udim \ S_i, \udim  \ S_j \rangle_{kQ^{\rm op}}\\
= &\ \dim_k  \Ext^1_{kQ^{\rm op}}(S_i, S_j) -\dim_k  \Ext^1_{kQ}(S_i, S_j).
\end{align*}
Now recall the entry in the $i^{th}$ row and the $j^{th}$ column of the matrix $B_Q$ is given by the number of arrows in $Q$ from $i$ to $j$ minus the number of arrows in $Q$ from $j$ to $i$. Moreover, we have that $\dim_k  \Ext^1_{kQ}(S_i, S_j)$ coincides with the number of arrows from $j$ to $i$ in $Q$ (see for example \cite[ Lemma 2.12.b]{Assem}). The result follows.
\end{proof}

\subsection{Normal vectors of the facets}
We now describe the facets of the cones of $\Delta^+_{\seed}(\cX)$ in terms of their normal vectors. 

Since $kQ$ has finite global dimension there is a canonical triangulated equivalence  $D^b_{kQ} \to \KKb(\proj kQ)$ (resp. $D^b_{kQ} \to \KKb(\inj kQ)$) that maps a module to its minimal projective presentation (resp. to its minimal injective presentation). For an object $X $ of $D^b_{kQ}$ we let $P_X$ (resp. $I_X$) be the image of $X$ under this equivalence.
Since $\text{mod } kQ$ is hereditary, any module is quasi-isomorphic to its minimal projective presentation (resp. to its minimal injective presentation). Motivated by this observation, for an object $X$ of $D^b_{kQ}$ we define
\[
\gv^X_P:=\gv^{P_X}_P \quad \quad \text{and} \quad \quad \gv^X_I:=\gv^{I_X}_I.
\]
We denote by $\tau$ the Auslander--Reiten translation of $\text{mod }kQ$ (or $\KKb(\proj \ kQ)$ or $D^b_{kQ}$).

\begin{lemma}
\label{lem:mesh}
Let $Q$ be an acyclic quiver and $X$ be an object in $D^b_Q$. Then
\begin{enumerate}[i)]
    \item[$(i)$] $C_{kQ}^{-1} (\udim\; X)^T=\gv_P^{X}.$
    \item[$(ii)$] $(C_{kQ}^{-1})^T (\udim\; X)^T=-\gv_P^{\tau^{-1}X}.$
\end{enumerate}
\end{lemma}
\begin{proof}
Since the columns of $C_{kQ}$ are the dimension vectors of the indecomposable projective modules, we have that $ C_{kQ}(\gv^{P_i}_P)^T=\udim \ P_i$. Equivalently, $ C_{kQ}^{-1} (\udim \ P_i)^T=\gv^{P_i}_P$. 
Then $(i)$ follows from the fact that dimension vectors and the {\bf g}-vectors are additive in exact triangles and that $X$ is quasi-isomorphic to $P_X$.

We now show $(ii)$. By additivity, we can assume that $ X$ is an indecomposable module. Let
\[
0 \to X \to I^{\prime} \to I^{\prime\prime} \to 0
\]
be a minimal injective presentation of $X$. By definition of the Auslander-Reiten translation and since $ \text{mod }kQ$ is hereditary we have that 
\[0 \to \nu^{-1}I^{\prime} \to \nu^{-1}I^{\prime\prime} \to \tau^{-1}X \to 0.\]
is a minimal projective presentation of $\tau^{-1}X$, where $\nu$ is the Nakayama functor.
By additivity we obtain that 
\[
\gv_P^{\tau^{-1}X}=\gv_P^{\nu^{-1}I^{\prime\prime}}-\gv_P^{\nu^{-1}I^{\prime}}.
\]
Since $ \nu^{-1}I_i=P_i$ we have that $\gv^{\nu^{-1}I^{\prime}}_P=\gv^{I^{\prime}}_I $ and $\gv^{\nu^{-1}I^{\prime\prime}}_P=\gv^{I^{\prime\prime}}_I $. Moreover, by additivity we have that $\gv_I^{X}=\gv^{I^{\prime}}_I - \gv^{I^{\prime\prime}}_I $.
Putting these observations together we obtain that
\[
-\gv^{\tau^{-1}X}_{P}=\gv^{X}_I.
\]
Finally, the columns of $(C^{-1}_{kQ})^T$ are the dimension vectors of the indecomposable injective modules. Hence we can argue as in the preceding case that $(C_{kQ})^{-1}(\udim X)^T=\gv^X_I $. The result follows. 
\end{proof}

For every indecomposable complex $X$ of $D^b_{kQ}$ there is an Auslander-Reiten triangle 
\[
\label{eq:AR_triang}
X\to X_1\oplus \dots \oplus X_r \to \tau^{-1}X \to X[1],
\]
where the $X_1,\dots, X_n$ are indecomposable objects, see \cite[\S~3]{Hap}.
A mesh is a sequence of the form $X\to X_1\oplus \dots \oplus X_r \to \tau^{-1}X$; we write
\[
X_i\in \rmesh_X
\]
to indicate that $X_i$ is an indecomposable summand in the middle of such mesh.

\begin{theorem} \label{dv-gv}
Let $Q$ be an acyclic quiver and $X$ be an indecomposable object in $D^b_Q$, then
\[B_Q(\udim \; X)^T=-\gv^X_P -\gv^{\tau^{-1}X}_P=-\sum_{X_i\in \rmesh_X} \gv^{X_i}_P.\]
\end{theorem}

\begin{proof}
By Lemma~\ref{B-C}, we know that $B_Q=(C_{kQ}^{-1})^T-C_{kQ}^{-1}$. %Therefore, it is enough to show that
%\[ %(C_{kQ}^{-1})^T\cdot (\udim\; X)^T- C_{kQ}^{-1}\cdot (\udim\; X)^T=\sum_{X_i\in \rmesh_X} \gv^{X_i}_P. %\]
Using Lemma~\ref{lem:mesh} and the additivity of {\bf g}-vectors we obtain that
\[
B_Q (\udim\; X)^T=(C_{kQ}^{-1})^{T}(\udim\; X)^T-C_{kQ}^{-1}(\udim\; X)^T = -\gv^X_P -\gv^{\tau^{-1}X}_P= -\sum_{X_i\in \rmesh_X} \gv^{X_i}_P.
\] 
\end{proof}

\begin{example} \label{2Kronecker}
Let $Q$ be the $2$-Kronecker quiver
\[\begin{tikzcd}
Q: 1 && 2. \arrow[ll,shift left,""]
  \arrow[ll,shift right,swap,""]
\end{tikzcd}\]
So
\[B_Q=\begin{pmatrix}
 0 & -2 \\
2 & 0 
\end{pmatrix}.\]

We proceed to verify Theorem~\ref{dv-gv} in this case. Observe that it is enough to verify the statement for the stalk complexes concentrated at degree $0$. Indeed, for every $X\in D^b_{kQ}$ as $\udim \ X[1]=-\udim \ X$ and $\gv^{X[1]}_P=-\gv^X_P$. 

It is well-known that every indecomposable preprojective module in $kQ$ has a dimension vector of the form $(d, d+1)$, while every indecomposable preinjective module has a dimension vector of the form $(d+1, d)$, where $d\in \mathbb{N}$. Furthermore, indecomposable modules in the regular component of the Auslander-Reiten quiver of $kQ$ have dimension vectors of the form $(d+1, d+1)$.

Additionally, at the level of dimension vectors, the meshes in preprojective and preinjective components, $\mathcal{P}(A)$ and $\mathcal{Q}(A)$, of the Aulander-Reiten quiver in the derived category are of the following form, respectively 
\[
\begin{tikzcd}
(d,d+1) \arrow[rr,shift left]
  \arrow[rr,shift right,swap] && (d+1,d+2) \arrow[rr,shift left]
  \arrow[rr,shift right,swap] && (d+2,d+3)
\end{tikzcd}
\]
and 
\[
\begin{tikzcd}
(d+3,d+2) \arrow[rr,shift left]
  \arrow[rr,shift right,swap] && (d+2,d+1) \arrow[rr,shift left]
  \arrow[rr,shift right,swap] && (d+1,d)
\end{tikzcd}
\]
and in the regular component $\mathcal{R}(A)$ are of the form

\begin{center}
\begin{tikzpicture}[scale=0.8]

  \draw[dashed,color=gray] (0,1.5) ellipse (1.5 and 0.5);% ellipse 3
  \draw[dashed,color=gray] (0,3) ellipse (1.5 and 0.5);% ellipse 2
  \draw[dashed,color=gray] (0,4.5) ellipse (1.5 and 0.5);% ellipse 1
  \draw[color=gray] (-1.5,1.5) -- (-1.5,4.5);% left line
  \draw[color=gray] (1.5,1.5) -- (1.5,4.5);% right line

  \node (1) at (1.5,4.5) {$\bullet$};
  \node (5) at (2.1,4.5) {\hspace{1 cm} \small $(d+2,d+2)$};
  \node (2) at (-1.5,3) {$\bullet$};
  \node (6) at (-2.1,3) {\small $(d+1,d+1) \quad \quad \quad \quad$ };
  \node (3) at (1.5,1.5) {$\bullet$};  
  \node (7) at (2.1,1.5) {\small $(d,d)$};
  
  \draw (1) edge[->, >=latex, bend left=15] (2);
  \draw (2) edge[->, >=latex, bend left=15] (1);
  \draw (2) edge[->, >=latex, bend left=15] (3);
  \draw (3) edge[->, >=latex, bend left=15] (2);
\end{tikzpicture}
\end{center}

On the other hand, it can be observed that if $X$ is an indecomposable module with $\udim \; X = (d_1,d_2)$, $d_1,d_2\in \mathbb{N}$, then $\gv_P^X=(g_1,g_2)$, where $g_1=d_1$ and the value of $g_2$ depends on the specific form of $X$ as follows:

\[ g_2= \left\{ \begin{array}{lll}
-d_2+2, & \text{if} & X\in \mathcal{P}(A) \\
\\ -d_2-2, & \text{if} & X\in \mathcal{Q}(A) \\
\\ -d_2, & \text{if} & X \in \mathcal{R}(A) \\
\end{array}
\right.\]

Then, we have the following cases:

\begin{enumerate}
    \item for $X\in \mathcal{P}(A)$  with $\udim \ X = (d,d+1)$,  
\begin{align*}
    B_Q (\udim \ X) = & (-2(d+1),2d) = -2 (d+1,-(d+2)+2) =-2 \gv_P^{X^{\prime}},
\end{align*}
where $\udim \ X^{\prime} = (d+1,d+2)$. \par \bigskip 

    \item for $X\in \mathcal{Q}(A)$ with $\udim \; X = (d+2,d+1)$,
\begin{align*}
    B_Q (\udim \; X) = & (-2(d+1),2(d+2)) = -2 (d+1,-d-2) =-2 \gv_P^{X^{\prime}},
\end{align*}
where $\udim \ X^{\prime} = (d+1,d)$. \par \bigskip  

    \item for $X$ with $\udim \; X = (1,1)$, 
\begin{align*}
    B_Q (\udim \; X) = & (-2,2) = - \gv_P^{X^{\prime}},
\end{align*}
where $\udim \ X^{\prime} = (2,2)$. \par \bigskip  

    \item for $X$ in the regular component such that $\udim \; X = (d,d)$, $d>2$, 
\begin{align*}
    B_Q (\udim \; X) = & (-2d,2d) = -\big[(d+1,-d-1)+(d-1,-d+1)\big] =- [\gv_P^{X^{\prime\prime}}+\gv_P^{X^{\prime}}],
\end{align*}
where $\udim \ X^{\prime} = (d-1,d-1)$ and $\udim \ X^{\prime\prime} = (d+1,d+1)$. \par \bigskip  
\end{enumerate}
This illustrates Theorem~\ref{dv-gv} for all $kQ$-modules $X$. We can also see from these computations that the result holds for all $X\in D^b_{kQ}$. 
%since $\udim X[1]=-\udim X$ and $\gv^{X[1]}_P=-\gv^X_P$.
\end{example}

\subsection{Supporting hyperplanes via {\bf g}-vectors}
The aim of this subsection is to apply Theorem~\ref{dv-gv} to give the equation of every supporting hyperplane of a cone of $\Delta^+_{\seed}(\cX)$ using {\bf g}-vectors.

Using the description of the cone $(p^*)^{-1}(\mathcal G_{\seed'})$ through the system of inequalities
\[
C_{\seed'}^T B_Q \beta_{\seed} \geq \textbf{0},
\] 
as discussed in \S~\ref{sec:inequalities}, the following result tells us that a normal vector of the hyperplane of the form $c_{i;\seed'}^T B_Q \beta_{\seed} = 0$ can be described using {\bf g}-vectors of $kQ$-modules.

\begin{corollary} \label{Corollary-normalvectorhyperplane}
Suppose $c_{i;\seed'}$ is a positive $\cv$-vector. Let $X$ be the indecomposable rigid $kQ$-module such that $c_{i;\seed'}^T=\udim \ X$. Then the vector $\displaystyle{\sum_{X_j\in \rmesh_X}\gv_P^{X_j}}$ is normal to the hyperplane $%\alpha_{i;\seed}=
c_{i;\seed'}^T B_Q \beta_{\seed} = 0$.
%    The wall $\alpha_{i;\seed}=0$ determined by a positive $\cv$-vector $c_{i;\seed}=\udim(X)$ has as normal vector $\sum_{i}\gv_{X_i}$.
\end{corollary}
    
\begin{proof}
It is immediate that Theorem~\ref{dv-gv} implies that \begin{align*}
(-B_Q   (\udim \; X)^T)^{T}= & \bigg(\sum_{X_i\in \rmesh_X} \gv_P^{X_i}\bigg)^T    \\
(\udim \; X)   (-B_Q^{T})= & \sum_{X_i\in \rmesh_X} (\gv_P^{X_i})^T 
\end{align*}
Hence, we can deduce
\[ \alpha_{i}=c_{i;\seed'}^{T}  B_Q \beta_{\seed}=\sum_{X_j\in \rmesh_X}(\gv_P^{X_j})^T \beta_{\seed},\]
and the result follows.
\end{proof}

\begin{remark}\label{rem:relationgvectors}
The results of \S~\ref{dim} together with Corollary~\ref{Corollary-normalvectorhyperplane} imply that we can search for linear relations with positive coefficients between certain sets of {\bf g}-vectors to determine the implicit equalities of the system $C_{\seed'}^T B_Q \beta_{\seed} \geq \textbf{0}$.
% Assume that $X_i$ is an indecomposable rigid $kQ$-module corresponding such that $c_{i;\seed'}^T=\udim(X_i)$, for $i\in \{1,\dots,n\}$. Then we can apply the results of \S~\ref{dim} and Corollary~\ref{Corollary-normalvectorhyperplane} to find constrains in the dimension of $(p^{*})^{-1}(\mathcal G_{\seed'})$ by establishing relations among {\bf g}-vectors. That is if we consider the implicit equalities of $C_{\seed'}^T B_Q\cdot\beta_{\seed} \geq \textbf{0}$ then
More precisely, if $X_i$ is an indecomposable rigid $kQ$-module such that $c_{i;\seed'}^T=\udim \ X_i$, for $i\in \{1,\dots,n\}$ then
\[
\sum_{k\in H}\lambda_k(c_{k;\seed'}^T   B_Q \beta_{\seed})=0,
\]
 is equivalent to
\begin{equation}
 \label{eq:equiv}   
\sum_{k\in H}\lambda_k(\gv_P^{X_k}+\gv_P^{\tau^{-1}X_k})=\sum_{\substack{X_{j_k}\in \rmesh_{X_k} \\ k\in H}}\lambda_k\gv_P^{X_{j_k}}=\textbf{0}.
\end{equation}    
Hence, finding relations as in \eqref{eq:equiv} allows us to identify sets of implicit equations of the system.
\end{remark}

\begin{remark}
For each $\seed'\in \T_n$, all the hyperplanes $c_{i;\seed'}^{T}   B_Q   \beta_\seed=0$ can be determined through a ``mesh relation" inherited from the mesh relation satisfied by the corresponding dimension vectors in the Auslander-Reiten quiver in $\text{mod } kQ$. Consequently, the inequalities linked to the indecomposable projective modules assume a crucial role in deducing other inequalities related to $\cv$-vectors. For instance, for quivers of finite type, the inequalities associated with the indecomposable projective modules allow determining all the supporting hyperplanes of the cones.
\end{remark}

\begin{example}
All the hyperplanes associated with the $2$-Kronecker quiver in Example~\ref{2Kronecker} determined by a positive $\cv$-vector $c_{i;\seed'}^T=\udim(X)$, take the following forms:

\begin{enumerate}
    \item $(d+1)\beta_1-d\beta_2=0$, if $X\in \mathcal{P}(A)$, \par 
    
    \item $d\beta_1-(d+1)\beta_2=0$, if $X\in \mathcal{Q}(A)$, \par
    
    \item $\beta_1-\beta_2=0$, if $X$ is a regular module.
\end{enumerate}
where $d\in\mathbb{N}$. These hyperplanes can be positioned in the Auslander-Reiten quiver according to the location of their respective $\cv$-vectors (see. Figure~\ref{fig:Hyperplanes_Auslander_Reiten_quiver}).

\begin{figure}[H]
\begin{center}
\begin{tikzpicture}[scale=0.75]

  \draw[dashed,color=gray] (0,1.5) ellipse (1 and 0.5);% ellipse 3
  \draw[dashed,color=gray] (0,3) ellipse (1 and 0.5);% ellipse 2
  \draw[dashed,color=gray] (0,4.5) ellipse (1 and 0.5);% ellipse 1
  \draw[color=gray] (-1,1.5) -- (-1,5.5);% left line
  \draw[color=gray] (1,1.5) -- (1,5.5);% right line
  \node (4) at (0,5.5) {\small $\vdots$};
    
  \node (1) at (1,4.5) {$\bullet$};
  \node (5) at (2.1,4.5) {\small $\textcolor{green}{\mathbf{3\beta_1-3\beta_2}}$};
  \node (2) at (-1,3) {$\bullet$};
  \node (6) at (-1.5,3) {\small $\textcolor{green}{\mathbf{2\beta_1-2\beta_2}} \quad \quad \quad$ };
  \node (3) at (1,1.5) {$\bullet$};  
  \node (7) at (2.3,1.5) {\small $\textcolor{green}{\mathbf{\beta_1-\beta_2}}$};
  
  \draw (1) edge[->, >=latex, bend left=15] (2);
  \draw (2) edge[->, >=latex, bend left=15] (1);
  \draw (2) edge[->, >=latex, bend left=15] (3);
  \draw (3) edge[->, >=latex, bend left=15] (2);

%preprojective 

\node (8) at (-5,3.6) {$\bullet$};
  \node (9) at (-5,4) {\small $\textcolor{blue}{4\beta_1-3\beta_2}$ };
\node (10) at (-6.5,2.4) {$\bullet$}; 
  \node (11) at (-6.5,2) {\small $\textcolor{blue}{3\beta_1-2\beta_2}$};
\node (12) at (-8,3.6) {$\bullet$};
  \node (13) at (-8,4) {\small $\textcolor{blue}{2\beta_1-\beta_2}$ };  
\node (14) at (-9.5,2.4) {$\bullet$}; 
  \node (15) at (-9.5,2) {\small $\textcolor{blue}{\beta_1}$};
\node (16) at (-4.5,3) {$\dots$}; 

\draw (14) edge[->, >=latex, bend left=15] (12);
\draw (14) edge[->, >=latex, bend right=15] (12);
\draw (12) edge[->, >=latex, bend left=15] (10);
\draw (12) edge[->, >=latex, bend right=15] (10);
\draw (10) edge[->, >=latex, bend left=15] (8);
\draw (10) edge[->, >=latex, bend right=15] (8);
  
%preinjective 
 
\node (17) at (9.5,3.6) {$\bullet$};
  \node (18) at (9.5,4) {\small $\textcolor{red}{-\beta_2}$ };
\node (19) at (8,2.4) {$\bullet$}; 
  \node (20) at (8,2) {\small $\textcolor{red}{\beta_1-2\beta_2}$};
\node (21) at (6.5,3.6) {$\bullet$};
  \node (22) at (6.5,4) {\small $\textcolor{red}{2\beta_1-3\beta_2}$ };  
\node (23) at (5,2.4) {$\bullet$}; 
  \node (24) at (5,2) {\small $\textcolor{red}{3\beta_1-4\beta_2}$};
\node (25) at (4.5,3) {$\dots$};

\draw (23) edge[->, >=latex, bend left=15] (21);
\draw (23) edge[->, >=latex, bend right=15] (21);
\draw (21) edge[->, >=latex, bend left=15] (19);
\draw (21) edge[->, >=latex, bend right=15] (19);
\draw (19) edge[->, >=latex, bend left=15] (17);
\draw (19) edge[->, >=latex, bend right=15] (17);
\end{tikzpicture}
\end{center}
    \caption{Supporting hyperplanes in the Auslander-Reiten quiver in $\text{mod } kQ$.}    \label{fig:Hyperplanes_Auslander_Reiten_quiver}
\end{figure}
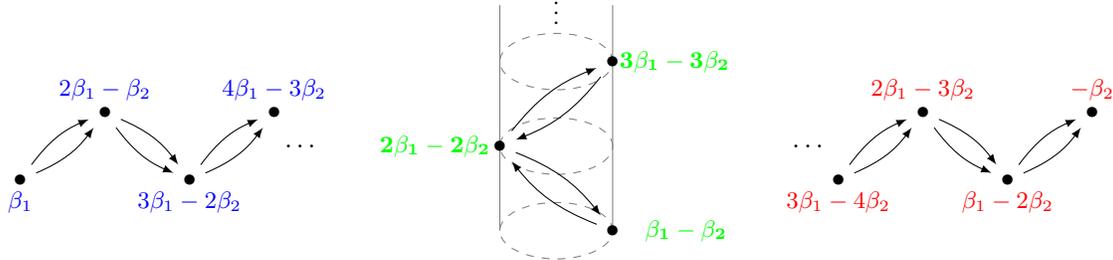

In this way, the supporting hyperplanes can be visualized in Figure~\ref{fig:Supporting_hyperplanes_kronecker}.

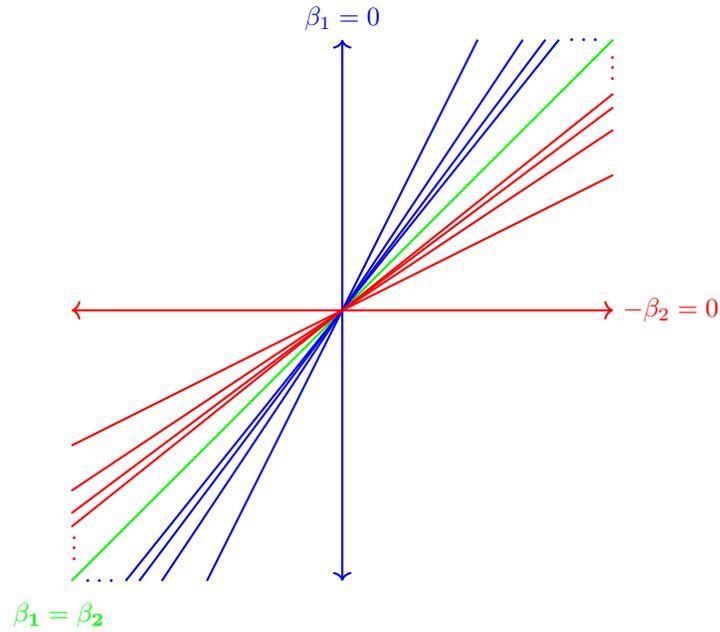
\begin{figure}[H]
\begin{center}
\begin{tikzpicture}[scale=0.9]
\draw[red,thick,<->] (-4,0)--(4,0) node[right] {$-\beta_2=0$}; % Eje x
% Enumeración del eje x
%\foreach \x/\xtext in {-4/-4,-3/-3,-2/-2, -1/-1, 1/1, 2/2, 3/3, 4/4} 
%\draw[shift={(\x,0)},black] (0pt,2pt)--(0pt,-2pt) node[below] {$\xtext$};
%
% Enumeración del eje y
%\foreach \y/\ytext in {-4/-4,-3/-3,-2/-2,-1/-1, 1/1, 2/2, 3/3, 4/4} 
%\draw[shift={(0,\y)},black] (2pt,0pt)--(-2pt,0pt) node[left] {$\ytext$};
\draw[blue,thick,<->] (0,-4)--(0,4) node[left,above] {$\beta_1=0$}; % Eje y
\draw[green,thick,-] (4,4) -- (-4,-4); 
\node[color=green,right] at (-5,-4.5) {$\mathbf{\beta_1=\beta_2}$};

%\draw[red,thick,-] (-4,0) -- (4,0); 
%\node[color=red,right] at (4,0.5) {$\beta_2=0$};
%
%\draw[blue,thick,-] (0,4) -- (0,-4); 
%\node[color=blue,right] at (0.2,-4) {$\beta_1=0$};
%
\draw[blue,thick,-] (-2,-4) -- (2,4); 
\draw[blue,thick,-] (-8/3,-4) -- (8/3,4); 
\draw[blue,thick,-] (-3,-4) -- (3,4); 
\draw[blue,thick,-] (-16/5,-4) -- (16/5,4); 
\node[color=blue,right] at (-3.95,-4) {$\dots$};
\node[color=blue,right] at (3.2,4) {$\dots$};

\draw[red,thick,-] (-4,-2) -- (4,2); 
\draw[red,thick,-] (-4,-8/3) -- (4,8/3); 
\draw[red,thick,-] (-4,-3) -- (4,3); 
\draw[red,thick,-] (-4,-16/5) -- (4,16/5); 
\node[color=red,right] at (-4.15,-3.4) {$\vdots$};
\node[color=red,right] at (3.8,3.7) {$\vdots$};

\end{tikzpicture}
\end{center}
\caption{Supporting hyperplanes of $\underline{\Delta_{\seed}^+(\cX)}$.}
\label{fig:Supporting_hyperplanes_kronecker}
\end{figure}

%The colors blue, red and green represent whether the module belongs to $\mathcal{P}(A)$, $\mathcal{Q}(A)$ or $\mathcal{R}(A)$, respectively. 
Hyperplanes exhibiting a slope greater than $1$ (including the hyperplane with infinite slope) represent the modules belonging to $\mathcal{P}(A)$. The hyperplane with slope $1$ is associated with the regular component $\mathcal{R}(A)$. Hyperplanes that show slopes greater than or equal to 0 but less than 1 are related to the set $\mathcal{Q}(A)$.
Besides, by the relations between {\bf g}-vectors and $p^{*}$ obtained in Example~\ref{2Kronecker}, the behavior of the cones and the generating rays can be observed in Figure~\ref{fig:cluster_complex_kronecker}.
\begin{figure}[H]
\begin{center}
\begin{tikzpicture}[scale=0.75]
\draw[black,thick,<->] (-4.5,0)--(4.5,0) node[right,below] {}; 
\draw[black,thick,<->] (0,4.5)--(0,-4.5) node[right, below] {}; % Eje y
\node[color=black] at (0,-5) {$\underline{\Delta^+_{\seed}(\cA)}$};
\draw[green,thick,dashed,-] (0,0) -- (4,-4); 

\draw[gray,thick,-] (4,0) -- (4,0);
\draw[gray,thick,-] (0,4) -- (0,-4); 
\draw[gray,thick,-] (2,-4) -- (0,0); 
\draw[gray,thick,-] (8/3,-4) -- (0,0); 
\draw[gray,thick,-] (3,-4) -- (0,0); 
\draw[gray,thick,-] (16/5,-4) -- (0,0); 
\node[color=gray,right] at (3.1,-4) {$\dots$};

\draw[gray,thick,-] (4,-2) -- (0,0); 
\draw[gray,thick,-] (4,-8/3) -- (0,0); 
\draw[gray,thick,-] (4,-3) -- (0,0); 
\draw[gray,thick,-] (4,-16/5) -- (0,0); 
\node[color=gray,right] at (3.75,-3.4) {$\vdots$};
\node[color=black,right] at (5.2,0) {$\xrightarrow{(p^{*})^{-1}} \quad$};
\end{tikzpicture} \begin{tikzpicture}[scale=0.7]
\draw[black,thick,<->] (-4.5,0)--(4.5,0) node[right,below] {$\beta_1$}; 
\draw[black,thick,<->] (0,-4.5)--(0,4.5) node[left,above] {$\beta_2$}; % Eje y
\node[color=black] at (0,-5) {$\underline{\Delta_{\seed}^+(\cX)}$};
\draw[green,thick,dashed,-] (0,0) -- (-4,-4); 

\draw[gray,thick,-] (-4,0) -- (4,0);
\draw[gray,thick,-] (0,4) -- (0,-4); 
\draw[gray,thick,-] (-2,-4) -- (0,0); 
\draw[gray,thick,-] (-8/3,-4) -- (0,0); 
\draw[gray,thick,-] (-3,-4) -- (0,0); 
\draw[gray,thick,-] (-16/5,-4) -- (0,0); 
\node[color=blue,right] at (-4,-4) {$\dots$};

\draw[gray,thick,-] (-4,-2) -- (0,0); 
\draw[gray,thick,-] (-4,-8/3) -- (0,0); 
\draw[gray,thick,-] (-4,-3) -- (0,0); 
\draw[gray,thick,-] (-4,-16/5) -- (0,0); 
\node[color=gray,right] at (-4.2,-3.4) {$\vdots$};
\end{tikzpicture}
\end{center}
\caption{Cluster complexes $\underline{\Delta_{\seed}^+(\cA)}$ and $\underline{\Delta_{\seed}^+(\cX)}$.}
\label{fig:cluster_complex_kronecker}
\end{figure}
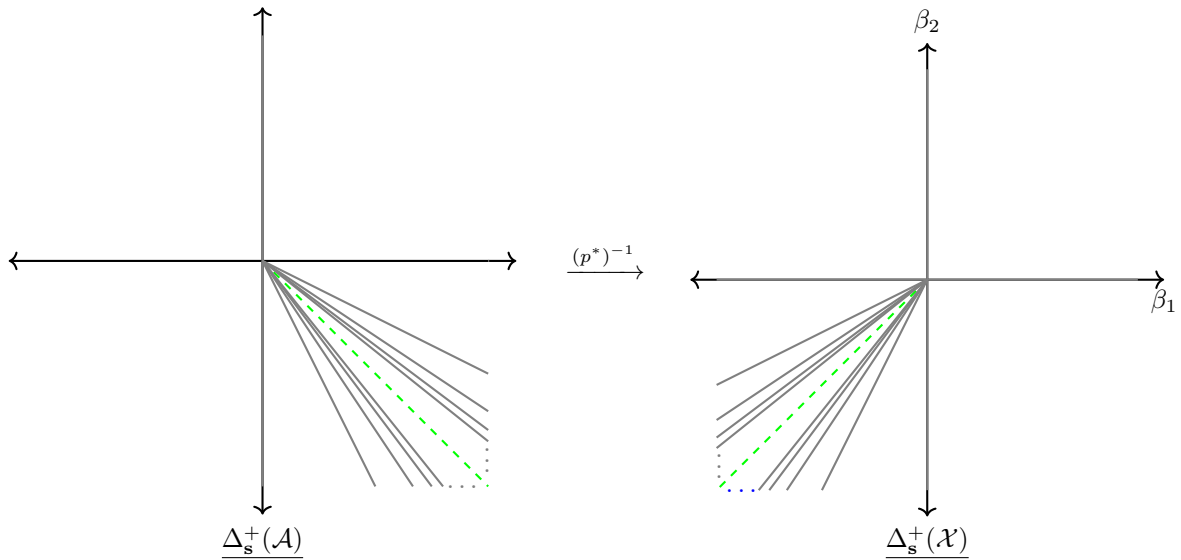
\end{example}

%\begin{remark}
%For cases involving negative $\cv$-vectors $c_{i;\seed'}$, it suffices to consider the same expression as for its corresponding positive $\cv$-vector $-c_{i;\seed'}$, where $-c_{i;\seed'}^T=\udim(X)$. Then, in general,
%\[ 
% c_{i;\seed'}^{T}\cdot B_Q\cdot\beta_{\seed} = \sgn(c_{i;\seed'}) \sum_{X_j\in \rmesh_X}(\gv_P^{X_j})^T \cdot \beta_{\seed}, 
%\]
%where $\sgn(c_{i;\seed'})\in \{-1,1\}$.
%\end{remark}

\begin{example} \label{A3}
Consider the quiver $Q=1\rightarrow 2 \leftarrow 3$ and let $B_Q=\begin{psmallmatrix}
0 & 1 & 0 \\
-1 & 0 & -1 \\
0 & 1 & 0 
\end{psmallmatrix}$ be its associated matrix. We proceed to describe the cones $(p^*)^{-1}(\mathcal{G}_{\seed'})$ for all $\seed'$. It is convenient to work with triangulations of the $6$-gon to consider all the possible seeds.
We let  
\begin{center}
\begin{tikzpicture}[scale=1]
    
    \foreach \x in {30,90,...,330} {
        \draw[fill] (\x:1 cm) -- (\x +60:1 cm);
      } 
      \draw[fill][red] (30:1 cm) -- node[black, anchor=east] {$\bullet$} (150:1 cm);
      \draw[fill][red] (30:1 cm) -- node[black] {$\bullet$} (210:1 cm);
      \draw[fill][red] (210:1 cm) -- node[black, anchor=west] {$\bullet$} (330:1 cm);

\draw[->, >=latex] (-0.2,0.5) -- (-0.03,0.1);
\draw[->, >=latex] (0.2,-0.45) -- (0.03,-0.1);
\end{tikzpicture}
\end{center} 
be the triangulation associated to $\seed$.
We proceed to verify that the seeds $\seed'$ such that $(p^*)^{-1}(\mathcal{G}_{\seed'})$ is $3$-dimensional are precisely the bipartite seeds. We also identify those seeds that give rise to $2$ and $1$-dimensional cones. 

In the case of seeds corresponding to bipartite orientations of $\mathbb{A}_3$, they generate the maximal cones, as no inequalities belong to $[C_{\seed'}^TB_Q]^=\beta_{\seed} \geq \textbf{0}^{=}$ as shown in Table~\ref{tab:3A3}.
\begin{table}[H]
\begin{center}
\begin{tabular}{| c | c | c | c | c | c | c |}
\hline
Triangulation & \begin{tikzpicture}[scale=0.8]
    
    \foreach \x in {30,90,...,330} {
        \draw[fill] (\x:1 cm) -- (\x +60:1 cm);
      } 
      \draw[fill][red] (30:1 cm) -- node[black, anchor=east] {$\bullet$} (150:1 cm);
      \draw[fill][red] (30:1 cm) -- node[black] {$\bullet$} (210:1 cm);
      \draw[fill][red] (210:1 cm) -- node[black, anchor=west] {$\bullet$} (330:1 cm);

\draw[->, >=latex] (-0.25,0.5) -- (-0.04,0.1);
\draw[->, >=latex] (0.25,-0.45) -- (0.04,-0.1);
\end{tikzpicture} & 

\begin{tikzpicture}[scale=0.8]
    
    \foreach \x in {30,90,...,330} {
        \draw[fill] (\x:1 cm) -- (\x +60:1 cm);
      } 
      \draw[fill][red] (90:1 cm) -- (210:1 cm);
      \draw[fill][red] (30:1 cm) -- node[black] {$\bullet$} (210:1 cm);
      \draw[fill][red] (30:1 cm) -- (270:1 cm);

\node (1) at (-0.3,0.5) {$\bullet$};
\node (2) at (0.3,-0.5) {$\bullet$};
         
\draw[->, >=latex] (0,0) -- (-0.27,0.47);
\draw[->, >=latex] (0,0) -- (0.27,-0.47);

\end{tikzpicture}
& 

\begin{tikzpicture}[scale=0.8]
    
    \foreach \x in {30,90,...,330} {
        \draw[fill] (\x:1 cm) -- (\x +60:1 cm);
      } 
      \draw[fill][red] (90:1 cm) -- (210:1 cm);
      \draw[fill][red] (90:1 cm) -- node[black] {$\bullet$} (270:1 cm);
      \draw[fill][red] (30:1 cm) -- (270:1 cm);
      
\node (1) at (-0.65,-0.1) {$\bullet$};
\node (2) at (0.65,0.1) {$\bullet$};
         
\draw[<-, >=latex] (-0.03,0) -- (-0.62,-0.08);
\draw[<-, >=latex] (0.03,0) -- (0.62,0.08);
\end{tikzpicture}
& 
\begin{tikzpicture}[scale=0.8]
    
    \foreach \x in {30,90,...,330} {
        \draw[fill] (\x:1 cm) -- (\x +60:1 cm);
      } 
      \draw[fill][red] (90:1 cm) -- (330:1 cm);
      \draw[fill][red] (90:1 cm) -- node[black] {$\bullet$} (270:1 cm);
      \draw[fill][red] (150:1 cm) -- (270:1 cm);
      
\node (1) at (-0.65,0.1) {$\bullet$};
\node (2) at (0.65,-0.1) {$\bullet$};
         
\draw[->, >=latex] (0,0) -- (-0.56,0.08);
\draw[->, >=latex] (0,0) -- (0.56,-0.08);
\end{tikzpicture}
&
\begin{tikzpicture}[scale=0.8]
    
    \foreach \x in {30,90,...,330} {
        \draw[fill] (\x:1 cm) -- (\x +60:1 cm);
      } 
      \draw[fill][red] (90:1 cm) -- (330:1 cm);
      \draw[fill][red] (150:1 cm) -- node[black] {$\bullet$} (330:1 cm);
      \draw[fill][red] (150:1 cm) -- (270:1 cm);
      
\node (1) at (-0.3,-0.5) {$\bullet$};
\node (2) at (0.3,0.5) {$\bullet$};
         
\draw[<-, >=latex] (-0.03,-0.03) -- (-0.27,-0.47);
\draw[<-, >=latex] (0.03,0.03) -- (0.27,0.47);
\end{tikzpicture}
& 
\begin{tikzpicture}[scale=0.8]
    
    \foreach \x in {30,90,...,330} {
        \draw[fill] (\x:1 cm) -- (\x +60:1 cm);
      } 
      \draw[fill][red] (30:1 cm) -- node[black, anchor=west] {$\bullet$} (150:1 cm);
      \draw[fill][red] (150:1 cm) -- node[black] {$\bullet$} (330:1 cm);
      \draw[fill][red] (210:1 cm) -- node[black, anchor=east] {$\bullet$} (330:1 cm);
               
\draw[->, >=latex] (0,0) -- (-0.27,-0.47);
\draw[->, >=latex] (0,0) -- (0.27,0.47);
\end{tikzpicture}\\ \hline
 &  &  &  &  &  &\\
$\cv-$matrix & $\begin{psmallmatrix} 
1 & 0 & 0 \\
0 & 1 & 0 \\
0 & 0 & 1 
\end{psmallmatrix}$ & $\begin{psmallmatrix} 
-1 & 1 &  0 \\
 0 & 1 &  0 \\
 0 & 1 & -1 
\end{psmallmatrix}$ & $\begin{psmallmatrix} 
0 & -1 & 1 \\
1 & -1 & 1 \\
1 & -1 & 0 
\end{psmallmatrix}$ & $\begin{psmallmatrix} 
 0 & 0 & -1 \\
-1 & 1 & -1 \\
-1 & 0 &  0 
\end{psmallmatrix}$ & $\begin{psmallmatrix} 
 0 &  0 & -1 \\
 0 & -1 &  0 \\
-1 &  0 &  0 
\end{psmallmatrix}$ & $\begin{psmallmatrix} 
0 &  0 & 1 \\
0 & -1 & 0 \\
1 &  0 & 0 
\end{psmallmatrix}$ \\
&  &  &  &  &  &\\
\hline
& & & & & & \\
$C_{\seed'}^TB_Q \beta_{\seed} \geq \textbf{0}$ & $\begin{smallmatrix} 
\beta_2 \geq 0 \\
-\beta_1-\beta_3 \geq 0 \\
\beta_2 \geq 0
\end{smallmatrix}$ & $\begin{smallmatrix} 
-\beta_2 \geq 0 \\
-\beta_1+2\beta_2-\beta_3 \geq 0 \\
-\beta_2 \geq 0
\end{smallmatrix}$ & $\begin{smallmatrix} 
-\beta_1+\beta_2-\beta_3 \geq 0 \\
\beta_1-2\beta_2+\beta_3 \geq 0 \\
-\beta_1+\beta_2-\beta_3 \geq 0
\end{smallmatrix}$ & $\begin{smallmatrix} 
\beta_1-\beta_2+\beta_3 \geq 0 \\
-\beta_1-\beta_3 \geq 0 \\
\beta_1-\beta_2+\beta_3 \geq 0
\end{smallmatrix}$ & $\begin{smallmatrix} 
-\beta_2 \geq 0 \\
\beta_1+\beta_3 \geq 0 \\
-\beta_2 \geq 0
\end{smallmatrix}$ & $\begin{smallmatrix} 
\beta_2 \geq 0 \\
\beta_1+\beta_3 \geq 0 \\
\beta_2 \geq 0
\end{smallmatrix}$ \\ 
& & & & & & \\ \hline
\end{tabular}
\caption{3-dimensional cones $\mathbb{A}_3$.}
\label{tab:3A3}
\end{center}
\end{table}
Additionally, for each bipartite seed $\seed'$, we observe that $c_{1;\seed'}  B_Q  \beta_{\seed} = c_{3;\seed'} B_Q  \beta_{\seed}$, hence the subset $\{c_{1;\seed'} B_Q  \beta_{\seed} \geq 0,  c_{2;\seed'} B_Q  \beta_{\seed}\geq 0\}$ of $[C_{\seed'}^TB_Q]^{>}\beta_{\seed} \geq \textbf{0}^{>}$ is not redundant. Consequently, $(p^{*})^{-1}(\mathcal G_{\seed'})$ has two facets:
\[
F_1 := (p^{*})^{-1}(\mathcal G_{\seed'}) \cap \{\beta \in N_{\R} \mid c_{1;\seed'}^{T} B_{\seed} \beta_{\seed} = 0\}, \quad F_2 := (p^{*})^{-1}(\mathcal G_{\seed'}) \cap \{\beta \in N_{\R} \mid c_{2;\seed'}^{T} B_{\seed} \beta_{\seed} = 0\}.
\]
Next, we consider the Auslander-Reiten quiver of $ D^b_{kQ}$:  

\begin{center}
\begin{tikzpicture}[scale=1]
    
\node (2) at (2,0) {100};
\node (3) at (4.2,0) {011};
\node (4) at (6.4,0) {$(001)[1]$};
\node (5) at (8.8,0) {$(110)[1]$};

\node (7) at (3.1,1) {111};
\node (8) at (5.3,1) {010};
\node (9) at (7.6,1) {$(111)[1]$};
\node (10) at (10,1) {$(010)[1]$};

\node (12) at (2,2) {001};
\node (13) at (4.2,2) {110};
\node (14) at (6.4,2) {$(100)[1]$};
\node (15) at (8.8,2) {$(011)[1]$};

\node (20) at (1,0) {$\cdots$};
\node (20) at (1,2) {$\cdots$};

\node (22) at (11.3,1) {$\cdots$};

\draw[->, >=latex] (2) -- (7);
\draw[->, >=latex] (7) -- (3);
\draw[->, >=latex] (3) -- (8);
\draw[->, >=latex] (8) -- (4);
\draw[->, >=latex] (4) -- (9);
\draw[->, >=latex] (9) -- (5);

\draw[->, >=latex] (12) -- (7);
\draw[->, >=latex] (7) -- (13);
\draw[->, >=latex] (13) -- (8);
\draw[->, >=latex] (8) -- (14);
\draw[->, >=latex] (14) -- (9);
\draw[->, >=latex] (9) -- (15);

\draw[->, >=latex] (15) -- (10);
\draw[->, >=latex] (5) -- (10);

\draw[red, dashed, rounded corners] (1.5, -0.5) rectangle (2.5, 2.5) {};

\draw[red, dashed, rounded corners] (3.75, -0.5) rectangle (4.7, 2.5) {};

\draw[red, dashed, rounded corners] (5.75, -0.5) rectangle (7.05, 2.5) {};

\draw[red, dashed, rounded corners] (8.2, -0.5) rectangle (9.45, 2.5) {};
\end{tikzpicture}
\end{center}
We can use Theorem~\ref{dv-gv} to see that the difference between each pair of elements enclosed within the same red rectangle belongs to the kernel of $p^{*}$. Consequently, a $\cv$-matrix containing one of these two elements with positive sign and also containing the other element with negative sign generates a 2-dimensional cone of $\underline{\Delta^+_\seed(\cX)}$. 

Hence, 2-dimensional cones correspond to the seeds whose associated quiver is linearly oriented, as observed in the following table:

\begin{table}[H]
\begin{center}
\begin{tabular}{| c | c | c | c | c | c | c |}
\hline
Triangulation & \begin{tikzpicture}[scale=0.8]
    
    \foreach \x in {30,90,...,330} {
        \draw[fill] (\x:1 cm) -- (\x +60:1 cm);
      } 
      \draw[fill][red] (30:1 cm) -- node[black] {$\bullet$}(270:1 cm);
      \draw[fill][red] (90:1 cm) -- (270:1 cm);
      \draw[fill][red] (150:1 cm) -- node[black] {$\bullet$}(270:1 cm);

\node (1) at (0,0.4) {$\bullet$};          
\draw[->, >=latex] (0,0.4) -- (-0.35,-0.13);  
\draw[->, >=latex] (0.4,-0.2) -- (0.08,0.32);  

\end{tikzpicture} & 

\begin{tikzpicture}[scale=0.8]
    
    \foreach \x in {30,90,...,330} {
        \draw[fill] (\x:1 cm) -- (\x +60:1 cm);
      } 
      \draw[fill][red] (90:1 cm) -- node[black] {$\bullet$} (210:1 cm);
      \draw[fill][red] (90:1 cm) -- (270:1 cm);
      \draw[fill][red] (90:1 cm) -- node[black] {$\bullet$} (330:1 cm);

      \node (1) at (0,-0.5) {$\bullet$};
         
\draw[->, >=latex]  (-0.4,0.25) -- (-0.07,-0.4);  
\draw[->, >=latex] (0,-0.5) -- (0.37,0.17);  
\end{tikzpicture}
& 

\begin{tikzpicture}[scale=0.8]
    
    \foreach \x in {30,90,...,330} {
        \draw[fill] (\x:1 cm) -- (\x +60:1 cm);
      } 
      \draw[fill][red] (30:1 cm) -- node[black] {$\bullet$} (270:1 cm);
      \draw[fill][red] (30:1 cm) -- (210:1 cm);
      \draw[fill][red] (30:1 cm) -- node[black] {$\bullet$} (150:1 cm);

    \node (1) at (-0.4,-0.25) {$\bullet$};
         
\draw[->, >=latex] (0,0.5) -- (-0.35,-0.13);  
\draw[->, >=latex] (-0.4,-0.22) -- (0.32,-0.22); 
\end{tikzpicture}
& 
\begin{tikzpicture}[scale=0.8]
    
    \foreach \x in {30,90,...,330} {
        \draw[fill] (\x:1 cm) -- (\x +60:1 cm);
      } 
      \draw[fill][red] (90:1 cm) -- node[black] {$\bullet$} (210:1 cm);
      \draw[fill][red] (30:1 cm) -- (210:1 cm);
      \draw[fill][red] (210:1 cm) -- node[black] {$\bullet$} (330:1 cm);

    \node (1) at (0.4,0.25) {$\bullet$};
           
\draw[->, >=latex] (0.32,0.25) -- (-0.35,0.25);   
\draw[->, >=latex] (0,-0.5) -- (0.37,0.17);  

\end{tikzpicture}
&
\begin{tikzpicture}[scale=0.8]
    
    \foreach \x in {30,90,...,330} {
        \draw[fill] (\x:1 cm) -- (\x +60:1 cm);
      } 
      \draw[fill][red] (270:1 cm) -- node[black] {$\bullet$} (150:1 cm);
      \draw[fill][red] (330:1 cm) -- (150:1 cm);
      \draw[fill][red] (30:1 cm) -- node[black] {$\bullet$} (150:1 cm);

      \node (1) at (0.4,-0.25) {$\bullet$};
         
\draw[->, >=latex] (0.4,-0.2) -- (0,0.45);  
\draw[->, >=latex] (-0.5,-0.22) -- (0.32,-0.22);

\end{tikzpicture}
& 
\begin{tikzpicture}[scale=0.8]
    
    \foreach \x in {30,90,...,330} {
        \draw[fill] (\x:1 cm) -- (\x +60:1 cm);
      } 
      \draw[fill][red] (90:1 cm) -- node[black] {$\bullet$} (330:1 cm);
      \draw[fill][red] (150:1 cm) -- (330:1 cm);
      \draw[fill][red] (210:1 cm) -- node[black] {$\bullet$} (330:1 cm);

      \node (1) at (-0.4,0.25) {$\bullet$};
         
\draw[->, >=latex]  (-0.4,0.25) -- (-0.07,-0.4);           
\draw[->, >=latex] (0.4,0.25) -- (-0.3,0.25); 
\end{tikzpicture}\\ \hline
 &  &  &  &  &  &\\
$\cv-$matrix & $\begin{psmallmatrix} 
0 & -1 & 1 \\
-1 & 0 & 1 \\
-1 & 0 & 0 
\end{psmallmatrix}$ & $\begin{psmallmatrix} 
0 & 0 & -1 \\
1 & 0 & -1 \\
1 & -1 & 0 
\end{psmallmatrix}$ & $\begin{psmallmatrix} 
1 & 0 & 0 \\
0 & 1 & 0 \\
0 & 1 & -1 
\end{psmallmatrix}$ & $\begin{psmallmatrix} 
-1 & 1 & 0 \\
0 & 1 & 0 \\
0 & 0 & 1 
\end{psmallmatrix}$ & $\begin{psmallmatrix} 
0 & 0 & 1 \\
0 & -1 & 0 \\
-1 & 0 & 0 
\end{psmallmatrix}$ & $\begin{psmallmatrix} 
0 & 0 & -1 \\
0 & -1 & 0 \\
1 & 0 & 0 
\end{psmallmatrix}$ \\
&  &  &  &  &  &\\
\hline
& \multicolumn{2}{c|}{} & \multicolumn{2}{c|}{} & \multicolumn{2}{c|}{} \\
$C_{\seed'}^TB_Q \beta_{\seed} \geq \textbf{0}$ & \multicolumn{2}{c|}{$\begin{smallmatrix} 
\beta_1-\beta_2+\beta_3 \geq 0 \\
-\beta_2 \geq 0 \\
-\beta_1+\beta_2-\beta_3 \geq 0
\end{smallmatrix}$ } & \multicolumn{2}{c|}{$\begin{smallmatrix} 
\beta_2 \geq 0 \\
-\beta_1+2\beta_2-\beta_3 \geq 0 \\
-\beta_2 \geq 0
\end{smallmatrix}$} & \multicolumn{2}{c|}{$\begin{smallmatrix} 
-\beta_2 \geq 0 \\
\beta_1+\beta_3 \geq 0 \\
\beta_2 \geq 0
\end{smallmatrix}$} \\ 
& \multicolumn{2}{c|}{} & \multicolumn{2}{c|}{} & \multicolumn{2}{c|}{} \\ \hline

& \multicolumn{2}{c|}{} & \multicolumn{2}{c|}{} & \multicolumn{2}{c|}{} \\
$[C_{\seed'}^TB_Q]^{=}$ & \multicolumn{2}{c|}{$\begin{psmallmatrix} 
1 & -1 & 1 \\
-1 & 1 & -1 \\
\end{psmallmatrix}$ } & \multicolumn{2}{c|}{$\begin{psmallmatrix} 
0 & 1 & 0 \\
0 & -1 & 0 \\
\end{psmallmatrix}$} & \multicolumn{2}{c|}{$\begin{psmallmatrix} 
0 & -1 & 0 \\
0 & 1 & 0 \\
\end{psmallmatrix}$} \\ 
& \multicolumn{2}{c|}{} & \multicolumn{2}{c|}{} & \multicolumn{2}{c|}{} \\ \hline
\end{tabular}
\caption{2-dimensional cones $\mathbb{A}_3$.}
\label{tab:2A3}
\end{center}
\end{table}

 In accordance with \S~\ref{dim}, we can see that in each case considered in Table~\ref{tab:2A3}, inequalities $c_{1;\seed'} B_Q  \beta_{\seed}\geq 0$ and $c_{3;\seed'} B_Q  \beta_{\seed}\geq 0$ belong to the set of implicit equalities of $C_{\seed'}^TB_Q \beta_{\seed} \geq \textbf{0}$ and $\rank([C_{\seed'}^TB_Q]^{=})=1$. 
 %, therefore what is mentioned in the \S~\ref{dim} is fulfilled.  

 Next, the seeds giving rise to $1$-dimensional cones are those whose quiver is a cyclic orientation of a triangle. The corresponding seeds and the inequalities they induce are presented in Table~\ref{tab:A31}. In this case all the inequalities of the system are implicit equalities. Moreover, we have that $\rank([C_{\seed'}^TB_Q]^{=})=2$ and the one dimensional cone defined by these seeds is precisely the linear space spanned by the kernel of $p^{*}$.   

\begin{table}[H]
\begin{center}
\begin{tabular}{| c | c | c | c |}
\hline
Triangulation & $\cv-$matrix & $C_{\seed'}^TB_Q \beta_{\seed} \geq \textbf{0}$ & $[C_{\seed'}^TB_Q]^{=}$ \\ \hline
\begin{adjustbox}{trim=0 0 0 -0.15cm}\begin{tikzpicture}[scale=0.8]
        \foreach \x in {30,90,...,330} {
        \draw[fill] (\x:1 cm) -- (\x +60:1 cm);
      } 
      \draw[fill][red] (30:1 cm) -- node (1) [black] {$\bullet$} (270:1 cm);
      \draw[fill][red] (30:1 cm) -- node (2) [black] {$\bullet$}(150:1 cm);
      \draw[fill][red] (150:1 cm) -- node (3) [black] {$\bullet$}(270:1 cm);

\draw[->, >=latex] (0,0.5) -- (0.4,-0.2);  
\draw[->, >=latex] (0.32,-0.25) -- (-0.35,-0.25);
\draw[->, >=latex] (-0.4,-0.2) -- (-0.02,0.42); 
\end{tikzpicture} \end{adjustbox} & \begin{adjustbox}{trim=0 0 0 1.1cm}$\begin{psmallmatrix} 
0 & -1 & 0 \\
1 & -1 & 0 \\
0 & 0 & 1 
\end{psmallmatrix}$ \end{adjustbox} & \multirow{2}{*}{ $\begin{smallmatrix} 
\beta_2 \geq 0 \\
\beta_1-\beta_2+\beta_3 \geq 0 \\
-\beta_1-\beta_3 \geq 0
\end{smallmatrix}$ } & \multirow{2}{*}{$\begin{psmallmatrix} 
0 & 1 & 0 \\
1 & -1 & 1 \\
-1 & 0 & -1 \\
\end{psmallmatrix}$} \\ 
\begin{adjustbox}{trim=0 0 0 -0.15cm} \begin{tikzpicture}[scale=0.8]
    
    \foreach \x in {30,90,...,330} {
        \draw[fill] (\x:1 cm) -- (\x +60:1 cm);
      } 
      \draw[fill][red] (90:1 cm) -- node (1) [black] {$\bullet$} (210:1 cm);
      \draw[fill][red] (210:1 cm) -- node (2) [black] {$\bullet$} (330:1 cm);
      \draw[fill][red] (90:1 cm) -- node (3) [black] {$\bullet$} (330:1 cm);

\draw[->, >=latex] (0,-0.5) -- (-0.4,0.2);  
\draw[->, >=latex] (-0.5,0.25) -- (0.32,0.25);
\draw[->, >=latex] (0.4,0.2) -- (0.02,-0.44); 
\end{tikzpicture} \end{adjustbox}
 & \begin{adjustbox}{trim=0 0 0 1.1cm} $\begin{psmallmatrix} 
1 & 0 & 0 \\
0 & -1 & 1 \\
0 & -1 & 0 
\end{psmallmatrix}$ \end{adjustbox} &  & \\ \hline
\end{tabular}
\caption{1-dimensional cones $\mathbb{A}_3$.}
\label{tab:A31}
%\multicolumn{2}{c|}{}
\end{center}
\end{table}

The cluster complex $\underline{\Delta_{\seed}^+(\cX)}$ can be visualized in Figure~~\ref{fig:A3_1}.

\begin{figure}[H]
\begin{center}
\begin{tikzpicture}[scale=1]
\node (1) at (0,0) {\includegraphics[scale=0.45]{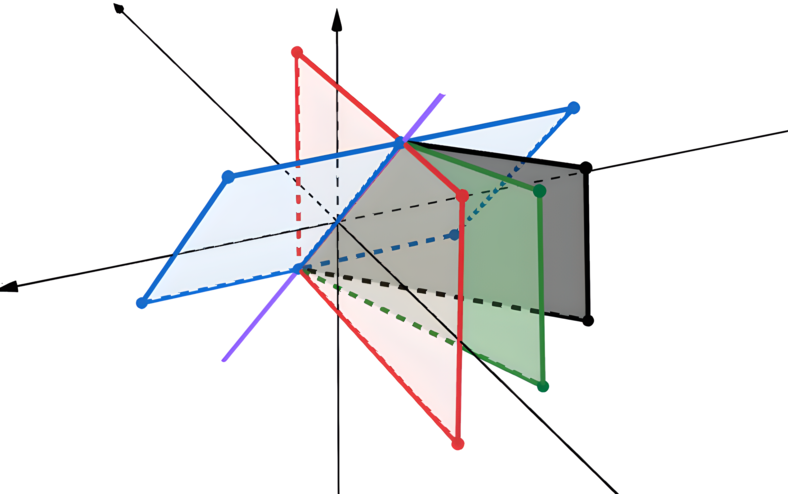}};
%{\includegraphics[scale=0.6]{A3B.png}};
\node at (-4,4) { $\beta_1$};
\node at (-6,-1) { $\beta_2$};
\node at (-0.5,3.7) { $\beta_3$};
\node[color=blue] at (4,2.5) { $-\beta_1-\beta_3=0$};
\node[color=black] at (5,-1.3) { $-\beta_1+\beta_2-\beta_3=0$};
\node[color=teal] at (4,-2.5) { $-\beta_1+2\beta_2-\beta_3=0$};
\node[color=red] at (1,-3.5) { $\beta_2=0$};
\node[color=violet] at (-3.3,-2.5) { $\ker(p^{*})=\Big\langle \begin{psmallmatrix}
 1 \\
 0 \\
-1  
\end{psmallmatrix} \Big\rangle$};
\end{tikzpicture}
\end{center}
    \caption{$\underline{\Delta_{\seed}^+(\cX)}$.}
    \label{fig:A3_1}
\end{figure}

%For a clearer perspective, the next image displays $(p^{*})^{-1}(\mathcal{G}_{\seed})$ after the quotient is taken by the linear space generated by $\ker(p^{*})$.
Since all the cones of $\underline{\Delta^+_{\seed}(\cX)}$ contain the linear space $\text{ker}(p^*)$ we have that the projection of $\underline{\Delta^+_{\seed}(\cX)}$ in $N_{\mathbb R}/\text{ker}(p^*)$ is a fan, in Figure~\ref{fig:A3_2} we display this fan.
More precisely, let $L$ be the orthogonal complement of $\ker(p^*)$ inside $N_{\R}$.
Then $N_{\R}/\ker(p^*)$ is canonically identified with $L$ and Figure~\ref{fig:A3_2} corresponds to the fan obtained by intersecting $\underline{\Delta^+_{\seed}(\cX) }$ with $L$. 
Furthermore, the image includes the {\bf c}-vectors and the seeds responsible for generating each cone of co-dimension one and maximal cone, respectively, as indicated within their corresponding rays and cones:

\begin{figure}[H]
\begin{center}
\begin{tikzpicture}[scale=0.8]
\draw[blue,thick,<-] (-6,0)--(6,0);
\node[color=blue] at (2,-0.2) {\tiny $-\beta_1-\beta_3=0$};

\node[color=blue] at (6.6,0) {$ \overline{\begin{psmallmatrix}
 0 \\
-1 \\
 0  
\end{psmallmatrix}}$};

\node[color=blue] at (-6.5,0) {$ \overline{\begin{psmallmatrix}
 0 \\
 1 \\
 0  
\end{psmallmatrix}}$};

\draw[red,thick,<-] (0,6)--(0,-6);
\node[color=red, rotate=-90] at (-0.3,-2) {\tiny $\beta_2=0$};

\node[color=red] at (0,6.5) {$\overline{\begin{psmallmatrix}
0 \\
0 \\
1  
\end{psmallmatrix}}=\overline{\begin{psmallmatrix}
1 \\
0 \\
0  
\end{psmallmatrix}}$};

\node[color=red] at (0,-6.5) {$\overline{\begin{psmallmatrix}
-1 \\
 0 \\
 0  
\end{psmallmatrix}}=\overline{\begin{psmallmatrix}
 0 \\
 0 \\
-1  
\end{psmallmatrix}}$}; % Eje y
\node[color=teal, rotate=-60] at (0.9,-2.2) {\tiny $-\beta_1+2\beta_2-\beta_3=0$};

\draw[teal,thick,-] (0,0) -- (3.2,-5.5);
\node[color=teal] at (3.3,-6.3) {$\overline{\begin{psmallmatrix}
-1 \\
-1 \\
-1  
\end{psmallmatrix}}$};

\draw[black,thick,-] (0,0) -- (5.5,-3);

\node[color=black, rotate=-29] at (1.9,-1.3) {\tiny $-\beta_1+\beta_2-\beta_3=0$};

\node[color=black] at (7,-3.5) {$\overline{\begin{psmallmatrix}
 0 \\
-1 \\
-1  
\end{psmallmatrix}}=\overline{\begin{psmallmatrix}
-1 \\
-1 \\
 0  
\end{psmallmatrix}}$};

\node[violet] at (0,0) { $\bullet$};

%:::::::::::::::::::::::::::::::::::::::::::::::

\node (1) at (-3,-3) {$\begin{psmallmatrix}
0 & 1 & 0 \\
-1 & 0 & -1 \\
0 & 1 & 0 \\
1 & 0 & 0 \\
0 & 1 & 0 \\
0 & 0 & 1 
\end{psmallmatrix}$};

\node (2) at (1.3,-4.95) {$\begin{psmallmatrix} 
0 & -1 & 0 \\
1 & 0 & 1 \\
0 & -1 & 0 \\
-1 & 1 & 0 \\
0 & 1 & 0 \\
0 & 1 & -1 
\end{psmallmatrix}$};

\node (3) at (3.7,-3.6) {$\begin{psmallmatrix}
0 & 1 & 0 \\
-1 & 0 & -1 \\
0 & 1 & 0 \\
0 & -1 & 1 \\
1 & -1 & 1 \\
1 & -1 & 0 
\end{psmallmatrix}$};

\node (4) at (4.6,-1) {$\begin{psmallmatrix} 
0 & -1 & 0 \\
1 & 0 & 1 \\
0 & -1 & 0 \\
0 & 0 & -1 \\
-1 & 1 & -1 \\
-1 & 0 & 0 
\end{psmallmatrix}$};

\node (5) at (3,3) {$\begin{psmallmatrix}
0 & 1 & 0 \\
-1 & 0 & -1 \\
0 & 1 & 0 \\
0 & 0 & -1 \\
0 & -1 & 0 \\
-1 & 0 & 0 
\end{psmallmatrix}$};

\node (6) at (-3,3) {$\begin{psmallmatrix} 
0 & -1 & 0 \\
1 & 0 & 1 \\
0 & -1 & 0 \\
0 & 0 & 1 \\
0 & -1 & 0 \\
1 & 0 & 0 
\end{psmallmatrix}$};

\end{tikzpicture}
\end{center}
    \caption{The fan in $ N_{\R}/\text{ker}(p^*)$ induced by $\underline{\Delta^+_{\seed}(\cX)}$.}
    \label{fig:A3_2}
\end{figure}

By Corollary~\ref{Corollary-normalvectorhyperplane}, we can see the normal vector of each of the walls in the Auslander-Reiten quiver given by the {\bf g}-vectors.

\begin{center}
\begin{tikzpicture}[scale=1]
    
\node (2) at (2,0) {100};
\node (3) at (4,0) {-110};
\node (4) at (6,0) {00-1};

\node (7) at (3,1) {010};
\node (8) at (5,1) {-11-1};
\node (9) at (7,1) {0-10};

\node (12) at (2,2) {001};
\node (13) at (4,2) {01-1};
\node (14) at (6,2) {-100};

\draw[->, >=latex] (2) -- (7);
\draw[->, >=latex] (7) -- (3);
\draw[->, >=latex] (3) -- (8);
\draw[->, >=latex] (8) -- (4);
\draw[->, >=latex] (4) -- (9);

\draw[->, >=latex] (12) -- (7);
\draw[->, >=latex] (7) -- (13);
\draw[->, >=latex] (13) -- (8);
\draw[->, >=latex] (8) -- (14);
\draw[->, >=latex] (14) -- (9);

\draw[red, dashed, rounded corners] (2.55, 0.5) rectangle (3.45, 1.5) {};

\draw[teal, dashed, rounded corners] (3.55, -0.5) rectangle (4.45, 2.5) {};

\draw[black, dashed, rounded corners] (4.55, 0.5) rectangle (5.45, 1.5) {};

\draw[blue, dashed, rounded corners] (5.55, -0.5) rectangle (6.45, 2.5) {};
\end{tikzpicture}
\end{center}

\end{example}

\begin{remark}
Although in case $\mathbb{A}_3$ the maximal cones only correspond to the bipartite seed this is not true in general. In fact, for $k\geq 2$, seeds associated with quivers $\mathbb{A}_{2k+1}$ whose orientation is not bipartite may also determine maximal cones of $\underline{\Delta^+_{\seed}(\cX)} $.
\end{remark}

\section{Acknowledgments}
This work was supported by Dirección General de Asuntos del Personal Académico, Universidad Nacional
Autónoma de México 2022, PAPIIT project
[IA100122 to C.M. and A.N.C]; and the Consejo Nacional de Humanidades, Ciencias y Tecnologías (CONAHCYT)
[CF-2023-G-106 to A.N.C.].

\noindent{\sc{Carolina Melo\\
Departamento de Matem\'aticas, Universidad Nacional de Colombia - Sede Bogot\'a, Ave Cra 30 \#45-3, Bogot\'a, Colombia}}\\
{\it{e-mail:}} \href{mailto:amelol@unal.edu.co}{amelol@unal.edu.co} \medskip

\noindent{\sc{Alfredo N\'ajera Ch\'avez\\
Instituto de Matem\'aticas UNAM Unidad Oaxaca, Le\'on 2, altos, Oaxaca de Ju\'arez, Centro Hist\'orico, 68000 Oaxaca, Mexico}}\\
{\it{e-mail:}} \href{mailto:najera@im.unam.mx}{najera@im.unam.mx}

\end{document}